\documentclass{SLC}

\articlenumber{9}
\newcommand{\qbin}[2]{{\scriptscriptstyle{\begin{bmatrix}{#1}\\ {#2}\end{bmatrix}_q}}}
\newcommand{\qbinm}[2]{{\scriptscriptstyle{\begin{bmatrix}{#1}\\ {#2}\end{bmatrix}_{\frac{1}{q}}}}}
\newcommand\dq{d_q}
\newcommand\smallminus{\text{-}}
\usepackage{tikz}
\author[Malvina Vamvakari]{Malvina Vamvakari\textsuperscript{1}\protect\orcid{0000-0002-1825-0097}}
\address{\textsuperscript{1}Harokopio University of Athens, Department of Informatics and Telematics, Athens, Greece; \website{https://dblp.org/pid/46/6916.html}}

\title[On $q$-order statistics]{On $q$-order statistics}

\abstract{Building on the notion of $q$-integral introduced by Thomae in 1869, 
we introduce $q$-order statistics (that, is $q$-analogues of the classical order statistics, for $0<q<1$)
which arise from dependent and not identically distributed $q$-continuous random variables and to study their distributional properties. 
We study the $q$-distribution functions and the $q$-density functions of the relative $q$-ordered random variables. 
We focus on $q$-ordered variables arising from dependent and not identically $q$-uniformly distributed random variables and we derive their $q$-distributions, including $q$-power law, $q$-beta and $q$-Dirichlet distributions.}
\keywords{$q$-order statistics, $q$-multinomial formulae, univariate and multivariate $q$-continuous random variables, $q$-uniform distribution, $q$-power law distribution, $q$-beta distribution, $q$-Dirichlet distribution, waiting times of the Heine process}
 
\begin{document}
\maketitle

\section{Introduction}
\label{briefintro}

Order statistics and their properties have been studied thoroughly the last decades. The literature devoted to  order statistics 
 from independent and identically distributed  random variables is very extensive. 
The study of order statistics arising from independent or dependent and not identically distributed,
 random variables, is of great research interest. Excellent references devoted to order statistics  are, among others, 
the  work of Arnold, Balakrishnan and Nagaraja~\cite{Arnold}, Balakrishnan~\cite{bala}, 
David and Nagaraja~\cite{David}, or Papadatos~\cite{papadatos1}.

In the field of discrete $q$-distributions, Charalambides~\cite[p.167]{Charal3} has presented the order statistics arising from independent and identically distributed random variables, with common distribution a discrete $q$-uniform distribution. 
 Charalambides~\cite{Charal1, Charal2} also has studied the distributions of the record statistics in $q$-factorially increasing populations.

The main objective of this work is to introduce $q$-order statistics, for $0<q<1$, arising from dependent and not identically distributed $q$-continuous random variables and to study their distributional properties. We introduce $q$-order statistics as $q$-analogues of the classical order statistics. 
We study the $q$-distribution functions and $q$-density functions of the relative $q$-ordered random variables. 
We focus on $q$-ordered variables arising from dependent and not identically $q$-uniformly distributed random variables 
and we derive their $q$-distributions, including $q$-power law, $q$-beta and $q$-Dirichlet distributions. Moreover, we consider the Heine process, which had been introduced by Kyriakoussis and Vamvakari~\cite{kyriakmvamv2}; see also the work of Kemp~\cite{Kemp1992}.
Note that our notion of $q$-distribution is not related to  the  $q$-Gaussian distribution,
or to other Tsallis distributions~\cite{tsallis}.
\pagebreak

We prove that a conditional $q$-joint distribution of the waiting times of the Heine process coincides with the joint $q$-density function of $q$-ordered random variables arising from dependent $q$-continuous random variables.

This work contains three sections along with the introductory Section~\ref{briefintro}. In the preliminary Section~\ref{prelimdefnot}, we present all our $q$-definitions.
In the main Section~\ref{main}, we state and prove our results concerning the $q$-order statistics and their distributional properties. 

\section{Preliminaries, definitions and notation}
\label{prelimdefnot}
In this section, we define the $q$-series, the univariate and multivariate $q$-continuous random variables,
the Heine process, and the $q$-uniform distribution. It will allow us to study $q$-order statistics in the next section.

\subsection{\texorpdfstring{$q$}{q}-Series preliminaries}

The $q$-shifted factorials are
\begin{equation}
 (a;q)_0 :=1, \quad (a;q)_n:=\prod_{k=1}^n(1-aq^{k-1}), \quad n=1,2,\ldots, \,\mbox{or $\infty$.}
\end{equation}
The multiple $q$-shifted factorials are defined by
\begin{equation*} (a_1,\ldots,a_k;q)_n:=\prod_{j=1}^k(a_j;q)_n\end{equation*}
 The $q$-binomial coefficient is defined by
\begin{equation*}\qbin{\nu}{k}=\frac{(q;q)_n}{(q;q)_k(q;q)_{n-k}}=\frac{[n]_{q}!}{[k]_{q}![n-k]_{q}},\,\,k=0,1,\ldots,n,
\end{equation*}
 where
\begin{equation*}
[n]_{q}! = [1]_{q}[2]_{q}\cdots[n]_{q}= \frac{(q;q)_n}{(1-q)^n} =\frac{\prod_{k=1}^n(1-q^k)}{(1-q)^n}\end{equation*}
is the $q$-factorial number of order $n$ with
 $[t]_{q}=\frac{1-q^{t}}{1-q}$.
\newline
The $k$th-order factorial of the number $[n]_q$ is called {\em $q$-factorial} of $n$ of order $k$ and is given by
\begin{equation*}
[n]_{k} =[n]_q[n-1]_q\cdots[n-k+1]_q, \;\;\; k=1,2,\ldots,n\; .
\end{equation*}
Note that
\begin{equation*}[n]_{q^{-1}}=q^{-n+1}[n]_q,
\,\,[n]_{q^{-1}}!=q^{-\binom{n}{2}}[n]_q!\,\,
\,\,\mbox{and}\quad \qbinm{n}{k}=q^{-k(n-k)}\qbin{n}{k}.
\end{equation*}
The $q$-binomial coefficient $\qbin{n}{k}$, equals the $k$-combinations $\{m_1,\ldots,m_k\}$ of the set $\{1,\ldots,n\}$, weighted by
 $q^{m_1+\cdots+m_k-\binom{k+1}{2}}$,
\begin{equation}
\label{combinations}
\sum_{1 \leq m_1<\cdots<m_k \leq n}q^{m_1+\cdots+m_k-\binom{k+1}{2}}= \qbin{n}{k}.
\end{equation}
Let $n$ be a positive integer and let $x,y$ and $q$ be real numbers, with $q \neq 1$. Then, a version of $q$-Vandermonde's 
 formula is 
\begin{equation}
\label{vandermonde}
 \qbin{x+y}{n}=\sum_{k=0}^n \qbin{n}{k} q^{k(y-n+k)}[x]_k  [y]_{ n-k}. \nonumber
\end{equation}
An interesting $q$-identity deduced by the above version of $q$-Vandermonde's formula is
\begin{equation}
\label{q-identity1}
\sum_{k=0}^n (-1)^k \qbin{n}{k} q^{\binom{k+1}{2}-n(y+k)}\frac{[y]_q}{[y+k]_q}=\frac{1}{ \qbin{y+n}{n}}.
\end{equation}
Note that from the above equation we have the corresponding $q^{-1}$-identity
\begin{equation}
\label{q-identity2}
\sum_{k=0}^n (-1)^k \qbin{n}{k} q^{\binom{k+1}{2}+ny}\frac{[y]_q}{[y+k]_q}=\frac{1}{\qbinm{y+n}{n}}.
\end{equation}
The $q$-binomial formula is 
\begin{equation*}\prod_{i=1}^n(1+tq^{i-1})=\sum_{k=0}^n q^{\binom{k}{2}} \qbin{n}{k} t^k.\end{equation*}
The above $q$-binomial formula, by replacing $q$ by $q^{-1}$ and $t$ by $-t$, becomes
\begin{equation}
\label{qbinom1}
\prod_{i=1}^n(1-tq^{-(i-1)})=\sum_{k=0}^n (-1)^k q^{-\binom{k}{2}} \qbinm{n}{k} t^k=\sum_{k=0}^n (-1)^k q^{-\binom{k}{2}-k(n-k)} 
\qbin{n}{k} t^k.
\end{equation}
The $q$-multinomial coefficient is defined for nonnegative integers $n$ and $k_i$'s by
\begin{equation}
\label{qmultinomial}
 \qbin{n}{k_1,\ldots,k_r}=\frac{[n]_{q}!}{[k_1]_{q} \cdots [k_{r}]_{q}![n-k_1-\cdots-k_{r}]_{q}!}.
\end{equation}
We then have the two following equivalent expressions for the $q^{-1}$-multinomial coefficient
\begin{align}
\label{qmultinomial2}
 \qbinm{n}{k_1,\ldots,k_r}&=q^{-\binom{n}{2}+\sum_{j=1}^{r+1}\binom{k_j}{2} }\qbin{n}{k_1,\ldots,k_r}\\&=q^{-\sum_{j=1}^r k_j\left(n-k_1-\cdots-k_j\right) }\qbin{n}{k_1,\ldots,k_r}.
\end{align}
An ordered set partition of $A$ is a sequence $(A_1,\dots, A_m)$ of non-empty disjoint subsets of~$A$, such that $A_1 \cup \dots \cup  A_m = A$.
Using the notation from Flajolet and Sedgewick's book \emph{Analytic Combinatorics}~\cite{FlajoletSedgewick2009}, 
ordered set partitions are accordingly defined by the symbolic formula $\operatorname{Seq}(\operatorname{Set}_{\geq 1})$,
and thus have the (exponential) generating function $1/(2-\exp(t))$ of Fubini numbers $\{F_n\}_{n\geq 0}=\{1,1,3,13,75,541,\dots\}$.
E.g., there are 13 ordered set partitions of \{1,2,3\}.

Charalambides~\cite{Charal4} showed that the $q$-multinomial coefficient $ \qbin{n}{k_1,\ldots,k_r}$ equals the number of 
ordered partitions of the set $\{1,\ldots,n \}$ into $r+1$ subsets, $\left( A_1,\ldots, A_{r+1} \right)$ of size  $(k_1,\dots, k_{r+1}$),
if one associates a specific $q$-weight to each subset.
Writing $A_j=\{m_{j,1},\ldots,m_{j,k_j}\}$, this weight is $q^{m_{j,1}+\cdots+m_{j,k_j}-\binom{k_j+1}{2}}$, and one has
\begin{align}
\label{qmultinomialseries}
\qbin{n}{k_1,\ldots,k_r}= \sum_{A_1,\dots,A_r} \prod_{j=1}^rq^{m_{j,1}+\cdots+m_{j,k_j}-\binom{k_j+1}{2}},
\end{align}
where the summation is over all the above mentioned ordered partitions of $\{1,\dots, n \}$.

Vamvakari~\cite{Malvina} earlier proved the following alternative summation expansion of the $q$-multinomial coefficient
\begin{align}
\label{qmultinomialseries2}
\sum\prod_{j=1}^rq^{k_{j,1}+2k_{j,2}+\cdots+n k_{j,n}-\binom{k_j+1}{2}}= \qbin{n}{k_1,\ldots,k_{r}},
\end{align}
where the summation is 
over all $k_{j,i}=0,1$ such that $\sum_{i=1}^n k_{j,i}=k_j$ (for $j=1,\ldots,r$).

We shall also use the following $q$-difference operator (which we also call ``$q$-derivative'')
\begin{equation}
\label{qdiff} \dq  f(x) :=
\frac{f(x)-f(qx)}{(1-q)x}. 
\end{equation}
We refer to~\cite[Chapter 10.2]{Andrews} or~\cite{Charal3} for a more thorough   discussion of its properties.
It is clear that it is a discrete analogue of the derivative; it satisfies e.g.
\begin{equation}
\dq  x^n=\frac{1-q^n}{1-q}x^{n-1}=[n]_q x^{n-1} \nonumber
\end{equation}
and $ \dq  (f(x)\cdot g(x))=g(x) \dq  f(x)+f(qx) \dq  g(x)$. What is more, for differentiable functions, one has
\begin{equation}
\lim_{q \rightarrow 1} \dq  f(x)=f'(x).\nonumber
\end{equation}   
Now, following~\cite[Chapter 10.1]{Andrews},  we define the $q$-integral  by
\begin{align}
\int_0^af(x)\dq x&:=\sum_{n=0}^\infty [aq^n-aq^{n+1}]f(aq^n),\label{qint:Fermat}\\
\int_a^bf(x)\dq x&:=\int_0^bf(x)\dq x-\int_0^af(x)\dq  x. \nonumber
\end{align}
In this context, $d_q$  is sometimes called the Fermat measure, and should not be confused with the above $q$-derivative, even if they are, in some sense, related.
The $q$-integral over $[0,\infty)$ uses the division points
$\{q^n:-\infty<n<\infty \}$ and is
\begin{equation}
\label{qintegr} \int_0^\infty
f(x)\dq x:=(1-q)\sum_{n=-\infty}^\infty q^n f(q^n).
\end{equation}

\subsection {The Heine process}

Kyriakoussis and Vamvakari~\cite{kyriakmvamv2} introduced the Heine process as a $q$-analogue of the Poisson process. The Heine process is defined as follows.

\begin{definition}[Heine Process]
\label{HeineProc}
A continuous time process $\{X(t),t>0\}$, where $X(t)$ expresses the number of arrivals in a time interval $(0,t]$, is called \emph{Heine process} with
parameters $0<q<1$ and $\lambda>0$, if the following three assumptions hold
\newline
(a) The process starts at time $0$ with $X(0)=0$.
\newline
(b) In each time interval of length $\delta=(1-q)t$,
one has 1 arrival with probability $p(t)$, and 0 arrival with probability $1-p(t)$, where
\begin{equation*}p(t):=\frac{\lambda (1-q)t}{1+\lambda(1-q)t}.\end{equation*}
That is,
\begin{equation*}
P\left( X(t)-X(q t)=1\right)=p(t)   \text{\qquad and   \qquad }  P\left( X(t)-X(q t)=0\right)=1-p(t).
\end{equation*}
\end{definition}
This implies, for any $k\geq 1$:
\begin{align}
P\left( X(q^{k-1}t)-X(q^k t)=1\right)=\frac{\lambda (1-q)q^{k-1}t}{1+\lambda(1-q)q^{k-1}t},\\ 
P\left( X(q^{k-1}t)-X(q^kt)=0\right)=\frac{1}{1+\lambda (1-q)q^{k-1}t}.
\end{align}
Also,   the Heine process has the Heine distribution:
\begin{equation}
\label{poiss_proc} P(X(t)=k)=e_q(-\lambda t)\frac{q^{\binom{k}{2}}(\lambda t)^k }{[k]_q!},
\end{equation}
for $k\in \mathbb N$, with $e_q(z)=\prod_{i=1}^\infty \left(1-(1-q)zq^{i-1}\right)^{-1}$, $|z|<1/(1-q)$.

\subsection{Univariate and multivariate \texorpdfstring{$q$}{q}-continuous random variables}
Kyriakoussis and Vamvakari~\cite{kyriakmvamv2} presented the following definition of  $q$-continuous random variables. 
For clarity, let us begin by presenting this concept for one random variable.
 \begin{definition}[$q$-continuous]
 \label{qcont}
 A random variable $X$ is called \emph{$q$-continuous}  (or ``Fermat integrable'', as we integrate over the Fermat measure defined in~\eqref{qint:Fermat})
if there exists a non-negative function $f_q(x)$ (for $x \geq  0$)
such that  \begin{equation*}P\left(\alpha<X \leq \beta\right)=\int_\alpha^\beta f_q(x) d_qx.\end{equation*}
 The function $f_q(x)$  is called \emph{$q$-density function} of the random variable $X$. 
 \end{definition}
Note that, in particular, one has
\begin{equation*}\int_0^\infty f_q(x)d_qx=1.\end{equation*}

 For the corresponding distribution function
 \begin{equation*} F(x)=P\left(X\leq x\right),\end{equation*}
 we have by definition
 \begin{equation*}P\left(\alpha<X\leq\beta\right)=F(\beta)-F(\alpha),\end{equation*}
and, for $x\geq0$,
 \begin{equation*}F(x)=\int_0^xf_q(t)d_qt.\end{equation*}
 Taking the $q$-derivative of the above relation we have
 \begin{equation*}\dq F(x)=f_q(x)\end{equation*}
 and by the definition of the $q$-derivative we obtain
 \begin{equation*}f_q(x)=\frac{F(x)-F(qx)}{(1-q)x}=\frac{P\left(qx<X \leq x\right)}{(1-q)x}.\end{equation*} 
 
Let us now present the case of tuples.
\begin{definition}[multivariate $q$-continuous]
\label{qcontmult}
A $k$-variate random variable ${\mathcal{X}}=(X_1,\ldots,X_k )$ is called
 \emph{$q$-continuous}  (or ``Fermat integrable'', as we integrate over the Fermat measure defined in~\eqref{qint:Fermat}) if there exists a non-negative function
$f_q(x_1,\ldots,x_k)$
 such that
 \begin{align}
 \label{multdef}
& P\left(\alpha_1<X_1 \leq \beta_1, \ldots, \alpha_k<X_k \leq \beta_k\right)=
\int_{\alpha_k}^{\beta_k} \cdots
\int_{\alpha_1}^{\beta_1}
f_q(x_1,\ldots,x_k)d_qx_1 \cdots d_qx_k.
 \end{align}
 The function $f_q(x_1,\ldots,x_k)$  is called \emph{$q$-density function}
 of the $k$-variate random variable ${\mathcal{X}}=(X_1,\ldots,X_k )$ or joint $q$-density
 function of the random variables
 $X_1,\ldots,X_k $.
\end{definition}
 In particular, we have
  \begin{equation*}\int_0^\infty \cdots \int_0^\infty f_q(x_1,\ldots,x_k)d_qx_1 \cdots d_qx_k=1.\end{equation*}
  
 For the corresponding joint distribution function
 \begin{equation*}F(x_1,\ldots,x_k)=P\left(X_1\leq x_1,\ldots,X_k \leq x_k\right)\end{equation*}
we have
 \begin{equation}\label{inteq}F(x_1,\ldots,x_k)=\int_{0}^{x_k} \cdots \int_{0}^{x_1} f_q(t_1,\ldots,t_k)d_qt_1 \cdots d_qt_k.\end{equation}
Building on the notation~\eqref{qdiff},
let us define the partial $q$-derivatives  by 
\begin{equation*}
\frac{\partial F(x_1,\ldots,x_k)}{\partial _qx_k \cdots \partial_qx_1}=(d _qx_k)  \cdots (d_qx_1)  F(x_1,\ldots,x_k).\end{equation*}

Then, taking the partial $q$-derivatives of the relation~\eqref{inteq},  we have 
 \begin{equation*}\frac{\partial F(x_1,\ldots,x_k)}{\partial _qx_k \cdots \partial_qx_1}=
 f_q(x_1,\ldots,x_k),\,\,x_i> 0,\,\,i=1,\ldots,k \end{equation*}
 and by the definition of the partial $q$-derivative we obtain
 \begin{equation}
 \label{multjointqdens}
 f_q(x_1,\ldots,x_k)=\frac{P\left(q x_1<X_1 \leq x_1, \ldots, q x_k<X_k \leq x_k\right)}{(1-q)x_1\cdots
(1-q)x_k}.
\end{equation}
 The marginal $q$-density functions of the random variables $X ,\,\,i=1,\ldots,k$, are given by
\begin{equation*} f_{X_{i}}(x_i)=\int_0^{\infty}\cdots \int_0^{\infty} \ f_{\mathcal{X}}(x_1,\ldots,x_k) d_qx_1\cdots d_qx_{i-1}d_qx_{i+1}\cdots d_qx_k,\,\,i=1,\ldots,k. \end{equation*}

For the needs of this work, we also define the conditional $q$-density function. 
Let $(X,Y)$ be a bivariate $q$-continuous random variable, with $q$-density function $f_q(x,y) \geq 0$, $x,y>0$ and $f_q(y)> 0$, $y>0$ the marginal $q$-density function of $Y$. 
Then the function 
\begin{equation*}f_{X|Y}(x|y)=\frac{f_{X,Y}(x,y)}{f_{Y}(y)}, x>0\end{equation*}
 is a $q$-density function because
\begin{equation*}f_{X|Y}(x|y)\geq 0, x>0\end{equation*}
and
\begin{equation*}\int_0^\infty f_{X|Y}(x|y)d_qx=\frac{1}{f_{Y}(y)} \int_0^\infty f_{X,Y}(x,y)d_qx=\frac{f_{Y}(y)}{f_{Y}(y)}=1. \end{equation*}
\pagebreak

Since
\begin{equation*}P\left(qx<X \leq x|qy<Y \leq y\right)=\frac{P\left(qx<X \leq x,qy<Y \leq y\right)}{P\left(qy<Y \leq	 y\right)}\end{equation*}
we confirm that 
\begin{equation}
\label{dependentqcont}
f_{X|Y}(x|y)=\frac{P\left(qx<X \leq x|qy<Y \leq y\right)}{(1-q)x}=\frac{\frac{P\left(qx<X \leq x,qy<Y<y\right)}{(1-q)x(1-q)y}}{\frac{P\left(qy<Y \leq y\right)}{(1-q)y}}=\frac{f_{X,Y}(x,y)}{f_{Y}(y)}
\end{equation}
and we give the following definition of conditional $q$-density function.
\begin{definition}[conditional $q$-density]
\label{dependentqjointdef}
Let $(X,Y)$ be a bivariate $q$-continuous random variable.
Let $f_{X,Y}(x,y)$ be its $q$-density function  and $f_Y(y)$ the marginal $q$-density function of $Y$. 
If $f_Y(y)>0$ for $y>0$,  the function
\begin{equation*}f_{X|Y}(x|y):=\frac{f_{X,Y}(x,y)}{f_{Y}(y)}\end{equation*}
 is called \emph{conditional $q$-density function} of the random variable $X$ given that $qy<Y<y$.
\end{definition}
Let $\left(X_1,\ldots,X_{k}\right)$ be a $q$-continuous $k$-variate random variable, with joint $q$-density function $f(x_1, \ldots,x_k)\geq 0$, $x_i>0$, $i=1,\ldots,k$. The {\it conditional $q$-density function} of a $q$-continuous $r$-variate random variable $\left(X_1,\ldots,X_{r}\right)$ given a $q$-continuous $(k-r)$-variate random variable $\left(X_{r+1},X_{r+2},\ldots,X_k \right)$ is expressed as
\begin{equation}
\label{multdependqdens}
h_{(X_1,\ldots,X_{r})|(X_{r+1},X_{r+2},\ldots,X_k )}(x_1,\ldots,x_r|x_{r+1},\ldots,x_k)=\frac{f(x_1,\ldots,x_r)}{g(x_{r+1},x_{r+2},\ldots,x_k)},
\end{equation}
where $g(x_{r+1},x_{r+2},\ldots,x_k)>0$ is the marginal $q$-density function of the $(k-r)$-variate random variable $\left(X_{r+1},X_{r+2},\ldots,X_k \right)$.

\subsection{On the \texorpdfstring{$q$}{q}-continuous \texorpdfstring{$q$}{q}-uniform distribution}
For the needs of this work we give the definition of the $q$-uniform distribution and derive easily its main characteristics and properties. The $q$-uniform distribution is defined as follows.
\begin{definition}[$q$-uniform] Let $X$ be a $q$-continuous random variable with $q$-density function
\begin{equation}
\label{qunifA}
f_q(x)= 
\begin{cases}
\frac{1}{\beta},& 0 \leq x \leq \beta, \\ 
0,& x<0 \mbox{ or } x>\beta,
\end{cases}
\end{equation}
where $\beta>0$. The distribution of the random variable $X$ is called $q$-uniform distribution with parameter $\beta$.
\end{definition}
Note that by the function~\eqref{qunifA} and the definition of the $q$-integral,
 \begin{equation*}
\int_0^\beta f_q(x)d_qx
=\sum_{n=0}^\infty \beta \left(q^n-q^{n+1}\right)
f_q(\beta q^n)
=1, \end{equation*}
as required by the definition of a $q$-density function.
\pagebreak

\begin{proposition}
\label{moments}
 The $r$-th $q$-moments of the $q$-uniform distribution is given by
\begin{equation}
\label{qmomentsA}
\mu_{r}=E\left(X^r\right)=\frac{\beta^{r}}{[r+1]_q}.
\end{equation}
In particular its $q$-mean and $q$-variance are given respectively by
\begin{equation}
\label{meanvar}
\mu_q=E\left(X\right)=\frac{\beta}{[2]_q }\,\,\, \mbox{and}\,\,\, \sigma_q^2=\frac{\beta^2 q}{(1+q+q^2)(1+2q+q^2)}.
\end{equation}
\end{proposition}
\begin{proof}
Using the $q$-density function~\eqref{qunifA} and the definition of the $q$-integral,
the $r$th $q$-moment of the $q$-uniform distribution,
\begin{equation}
\mu_{r}=E\left(X^r\right)=\int_0^\beta x^r f_q(x)d_qx,\nonumber
\end{equation}
is easily obtained in the form~\eqref{qmomentsA}. The $q$-mean and $q$-variance of $X$ follows.
\end{proof}
\begin{remark}
\label{compsition}
Let $X$ be a $q$-continuous random variable obeying a $q$-uniform distribution with parameter $\beta$, then the linearly transformed $q$-continuous random variable $Y=\frac{X}{\beta}$
obeys the $q$-uniform distribution with parameter $\beta=1$. Indeed
\begin{equation}
\label{transformation}
F_{Y}(y)=P\left(Y\leq y \right)=P\left(X \leq \beta y\right)=F_{X}\left(\beta y\right)
=\int_{\alpha}^{\beta y}f_q(x)d_qx=y,\,\,\,0 \leq y \leq1.
\end{equation}
So
\begin{equation*} f_{Y}(y)=
\begin{cases}
1, & 0 \leq y \leq 1,\\
0, & y<0 \mbox{ or } y>1.
\end{cases}
\end{equation*}
\end{remark}

In the following proposition, we show
that the linear transformation $Y=\frac{X}{\beta}$
can be generalized by considering the transformation $Y=F_{X}(X)$,
 where $F_{X}(x)$ is a distribution function of a $q$-continuous random variable $X$. 
\begin{proposition}
\label{generalcomposition}
Let $X$ be a $q$-continuous random variable with probability function $F_{X}(x)$, $x \in R$. Then the distribution of the $q$-continuous random variable $Y=F_{X}\left(X\right)$ is the $q$-uniform distribution with parameter $\beta=1$.
\end{proposition}
\begin{proof}
The distribution function of the $q$-continuous random variable $Y$ is given, for $0 \leq y \leq1$,  by
\begin{equation}
F_{Y}(y)=P\left(Y \leq y \right)=P\left(F_{X}(X) \leq y\right)=P\left(X\leq F_{X}^{-1}(y)\right)
=F_{X}\left(F_{X}^{-1}(y)\right)=y.
\end{equation}
So
\begin{equation*} f_{Y}(y)=
\begin{cases} 
1, & 0 \leq y \leq 1,\\
0, & y<0 \mbox{ or } y>1 
\end{cases} 
\end{equation*} 
and the proposition follows.
\end{proof}

\pagebreak

\section{Main results}
\label{main}

\subsection{\texorpdfstring{\!On the distributions of $q$-ordered random variables}{On the distributions of q-ordered random variables}}
Let a $\nu$-variate $q$-continuous random variable ${\mathcal{X}}=(X_1,\ldots,X_{\nu})$ be defined in a sample space~$\Omega$. Then for the values $x_1=X_1(\omega),\ldots,
x_{\nu}=X_{\nu}(\omega)$, $\omega \in \Omega$ there is a permutation $(i_1,\ldots, i_\nu)$ of the $\nu$ indices $\{ 1,\ldots,\nu \}$,
such that $x_{i_1}\leq \cdots <x_{i_{\nu-1}}\leq x_{i_\nu}$. 
The $k$-th ordered random variable is denoted by $X_{(k)}$ and defined by
\begin{equation*}X_{(k)}(\omega)=x_{(k)},\,\,\omega \in \Omega,\end{equation*}
where $x_{(k)}=x_{i_k}, \,k=1,\ldots,\nu$. In particular, for $k=1$ this gives $X_{(1)}=\min\{X_1,\cdots,X_{\nu}\}$ and, for $k=\nu$ this gives $X_{(\nu)}=\max\{X_1,\ldots,X_{\nu}\}$. Generally, the following inequalities hold:
 \begin{equation*}0 \leq X_{(1)}\leq \cdots \leq X_{(\nu)} \leq \beta,\end{equation*} 
 for a positive real number $\beta$.
 \smallskip
 
We now introduce the following definition of  $q$-ordered random variables.
\begin{definition}[$q$-ordered]
\label{qorderstat}
Let 
${\mathcal{Y}}=(Y_{1},\ldots,Y_{\nu})$ be a $\nu$-variate $q$-continuous random variable and $Y_{(k)},\, 1 \leq k \leq \nu$, be the corresponding $k$-th ordered random variables. 
Assume that $Y_{(k)},\, 1 \leq k \leq \nu$, satisfy the inequalities
\begin{equation}\label{qordered} 0 \leq Y_{(1)} <q Y_{(2)}  <  Y_{(2)} <    \cdots <Y_{(\nu-1)}< qY_{(\nu)}<Y_{(\nu)} \leq \beta, \end{equation}
for a positive real number $\beta$. Then, $Y_{(k)}$ (for any $k$ such that $1 \leq k \leq \nu$) is called the $k$-th \emph{$q$-ordered} random variables.
\end{definition}
\smallskip

Let  $Y_{(k)},\, 1 \leq k \leq \nu$, be the
 $k$-th $q$-ordered  random variables, where the non-ordered 
$q$-continuous random variables $Y_{1},\ldots,Y_{\nu}$, are {\it dependent and not identically distributed}.
The non-ordered, dependent and not identically distributed, random variables $Y_i, i=1, \dots,\nu$, are randomly selected from the same sample space
 and the corresponding $k$-th $q$-ordered random variables, $Y_{(k)}$, $1 \leq k \leq \nu$, satisfy inequalities~\eqref{qordered}.
 Each non-ordered random variable $Y_i$ is thus defined on the set 
\begin{align}
&R_{Y_i}:=[0,q^{(i-1)} \beta] = \cup_{j=i}^\nu R_j,  \\
& \text{where $R_j:=(q^j \beta, q^{j-1} \beta]$ for  $j=1, \ldots,\nu-1$ and  $R_\nu:=[0,q^{\nu-1} \beta]$}.    \label{sets}  
\end{align}
In particular, one has 
\begin{equation}
\cup_{j=1}^\nu R_j=[0, \beta] \text{ and  }  R_i \cap R_j=\emptyset \text{ for } i \neq j.
\end{equation}

 Moreover, we assume that the non-ordered random variables $Y_i$'s are not identically distributed according to their definitions sets 
 but they are distributed with the same functional form. Furthermore, the stochastic dependencies satisfied by the non-ordered random variables $Y_i$'s are explicitly defined hereafter.
\pagebreak

For any integer $r$ between 1 and $\nu$, let $\{i_1,\ldots,i_r\}$ be an $r$-combination of $\{1,\ldots,\nu \}$ 
satisfying $i_1<\cdots<i_r$, 
and let $\{i_{r+1},i_{r+2},\ldots,i_\nu\}$ be its complementary combination
(i.e., one has  $\{i_1,\ldots,i_r\} \cup \{i_{r+1},i_{r+2},\ldots,i_\nu\} = \{1,\ldots,\nu \}$) satisfying \mbox{$i_{r+1}<i_{r+2}<\cdots<i_\nu$}.
Then, we assume that the non-ordered random variables $Y_i$'s satisfy the following dependence relations
for $y \in [0,\beta]$:
 \begin{equation}
 \label{qdep1}
 P\left( Y_{i_r} \leq y|Y_{i_1} \leq y,\ldots, Y_{i_{r-1}}\leq y\right)=P\left( Y_{i_r} \leq q^{r-1} y \right), 
 \end{equation}
 \begin{equation}
 \label{qdep2}
 P\left( Y_{i_r} \leq y|Y_{i_1} > y, \ldots, Y_{i_{r-1}}> y\right)=P\left( Y_{i_r} \leq y \right), 
 \end{equation}
 and
 \begin{align}
 \label{qdep3}
&P\left( Y_{i_m} \leq y|Y_{i_1} \leq y,\ldots, Y_{i_{r}}\leq y, Y_{i_{r+1}}>y,\ldots, Y_{i_{m-1}} >y \right)\nonumber\\
 &\,\,\,\,\,\,\,\,\,\,\,\,=P\left( Y_{i_m} \leq y|Y_{i_1} \leq y,Y_{i_{r}}\leq y\right)\nonumber\\
 &\,\,\,\,\,\,\,\,\,\,\,\,= P\left( Y_{i_m, q} \leq q^{i_m-(m-r)} y \right), m=r+1,r+2,\ldots,\nu.
 \end{align}
The $q$-distribution functions of the maximum, minimum, and $k$-th $q$-ordered random variables
(respectively $Y_{(1)}$, $Y_{(\nu)}$, and $Y_{(k)}$) are derived in the following lemma.
\begin{lemma}
\label{qdistrib} Let $Y_{1},\ldots,Y_{\nu}$ be dependent $q$-continuous random variables, where 
\begin{itemize}
\item[(a)]  Each $Y_i$ is defined on the set $R_{Y_i}$ from Formula~\eqref{sets}.
\item[(b)] Each $Y_i$ has a $q$-distribution function $F_{Y_i}(y)$ $=P\left(Y_i\leq y\right)$, for $y \in R_{Y_i}$, 
of the same functional form and satisfies the dependence relations~\eqref{qdep1},~\eqref{qdep2},~\eqref{qdep3}. 
\end{itemize}
Then,  the $q$-distribution function of the maximum $q$-ordered random variable $Y_{(\nu)}=\max$ $\{{Y_{1}, \ldots},Y_{\nu}\}$,
 where $Y_{(i)}$, $i=1,\ldots,\nu, $ satisfy inequalities~\eqref{qordered}, 
 is given for $y \in [0, \beta]$ by
\begin{equation}
\label{max}
F_{Y_{(\nu)}}(y)=\prod_{i=1}^\nu F_{Y_i}(q^{i-1}y).
\end{equation}
Moreover, the $q$-distribution function of the minimum $q$-ordered random variable $Y_{(1)}=\min \{Y_{1},\ldots,Y_{\nu}\}$,
where $Y_{(i)}$, $i=1,\ldots,\nu, $ satisfy inequalities~\eqref{qordered},
 is given by
\begin{equation}
\label{min}
F_{Y_{(1)}}(y)=1-\prod_{i=1}^\nu \left(1-F_{Y_i}(y) \right).
\end{equation}
Furthermore, the $q$-distribution function of $k$-th $q$-ordered random variable $Y_{(k)}$, $1 \leq k \leq \nu$,
where $Y_{(i)}$, $i=1,\ldots,\nu, $ satisfy inequalities~\eqref{qordered},
 is given for $y \in [0, \beta]$ by
\begin{equation}
\label{general}
F_{Y_{(k)}}(y)=\sum_{r=k}^\nu \sum_{_{1\leq i_1<\ldots<i_r\leq \nu}} \prod_{j=1}^r F_{Y_{i_j}}\left(q^{j-1}y\right)\prod_{m=r+1}^\nu \left(1- F_{Y_{i_m}}\left( q^{i_m-(m-r)}y \right) \right), 
\end{equation}
where the inner summation is over all $r$-combinations $\{i_1, \ldots, i_r\}$ of the set $\{1, \ldots,\nu\}$.
\end{lemma}
\pagebreak
\begin{proof}
Let $ F_{Y_{(\nu)}}(y)$  be the $q$-distribution function of 
$Y_{(\nu)}=\max\{Y_{1},\ldots,Y_{\nu}\}$,
 then 
\begin{align}
\label{max1}
&F_{Y_{(\nu)}}(y)=P\left(Y_{(\nu)}\leq y\right)= P\left( \max \{Y_1,\ldots,Y_{\nu}\} \leq y \right)\nonumber\\&= P\left( Y_{1} \leq y,Y_{2} \leq y,\ldots,Y_{\nu} \leq y \right)
\nonumber\\&=P\left( Y_1 \leq y) P(Y_{2} \leq y \vert Y_{1} \leq y) \cdots P(Y_{\nu} \leq y \vert Y_{1} \leq y, 
\ldots, Y_{\nu-1} \leq y \right).
\end{align}
By assumptions (a) and (b), still for $y \in [0, \beta]$, the above equation~\eqref{max1} 
 becomes
 \begin{equation}
F_{Y_{(\nu)}}(y)=\prod_{i=1}^\nu F_{Y_i}(q^{i-1}y). \nonumber
\end{equation}
 Let also $ F_{Y_{(1)}}(y), y \in [0, \beta]$, be the $q$-distribution function of $Y_{(1)}=\min\{Y_{1},\ldots,Y_{\nu}\}$, then 
\begin{align}
\label{min1}
&F_{Y_{(1)}}(y)= P\left(Y_{(1)}\leq y\right)= 1- P\left(Y_{(1)}> y\right)=1-P\left( \min \{Y_{1},\ldots,Y_{\nu}\} > y \right) \nonumber\\ 
&=1- P\left( Y_{1} > y,Y_{2} > y,\ldots,Y_{\nu} > y \right)\nonumber\\
&=1- P\left( Y_1 > y) P(Y_2 > y \vert Y_1 >y) \cdots P(Y_{\nu} > y \vert Y_{1} >y, Y_{2} > y,\ldots, Y_{\nu-1} > y \right) \nonumber\\
&=1-\left(1- P\left( Y_1 < y)\right) (1-P(Y_{2}<y\vert Y_{1} >y) \cdots (1- P(Y_{\nu} < y \vert Y_{1}>y, \ldots, Y_{\nu-1}> y \right)). 
\end{align}
By assumptions (a) and (b), the above equation~\eqref{min1}
 becomes
 \begin{equation}
F_{Y_{(1)}}(y)= 1-\prod_{i=1}^\nu \left(1-F_{Y_i}(y) \right), y \in [0, \beta]. \nonumber
\end{equation}
Now, let $F_{Y_{(k)}}(y)$ be the $q$-distribution function of $Y_{(k)}$. Then, the event $Y_{(k)}\leq y$ occurs if and only if at least $k$ random variables from $\{Y_{1}, Y_\nu\}$ take values in the set $[0, y]$ while the remaining ones $\nu-k$ take values in the set $(y,\beta]$. More precisely, consider an $r$-combination $\{i_1,\ldots,i_r\}$ of $\{1,\ldots,\nu,\}$, with $i_1<\ldots<i_r$, and its complementary combination $\{i_{r+1},i_{r+2},\ldots,i_\nu\}$, with $i_{r+1}<i_{r+2}<\ldots<i_\nu$. Then the $q$-distribution function of $Y_{(k)}$ is expressed as
\begin{align}
\label{general1}
&F_{Y_{(k)}}(y)=P\left(Y_{(k)}\leq y\right)\nonumber\\
&=\sum_{r=k}^\nu\sum_{1\leq i_1<\ldots<i_{r}\leq \nu} P\left( Y_{i_1} \leq y,\ldots,Y_{i_{r}} \leq y, Y_{i_{r+1}} >y,\ldots, Y_{i_{\nu}} >y \right)\nonumber\\
&=\sum_{r=k}^\nu\sum_{1\leq i_1<\ldots<i_{r}\leq \nu} P\left( Y_{i_1} \leq y,\ldots,Y_{i_{r}} \leq y \right) P\left( Y_{i_{r+1}} >y,\ldots, Y_{i_{\nu}} >y| Y_{i_1} \leq y,\ldots,Y_{i_{r}} \leq y\right)\nonumber\\
&=\sum_{r=k}^\nu \sum_{_{1\leq i_1<\ldots<i_r\leq \nu}}P\left( Y_{i_1} \leq y) \cdots P(Y_{i_{r}} \leq y \vert Y_{i_1} \leq y,\ldots, Y_{i_{r-1}} \leq y \right)\nonumber\\
 &\phantom{=} \cdot 
 \resizebox{0.95\hsize}{!}{$
  P(Y_{i_{r+1}} > y \vert Y_{i_1} \leq y,\ldots, Y_{i_{r}} \leq y)\cdots P(Y_{i_{\nu}} > y \vert Y_{i_1} \leq y,\ldots, Y_{i_{r}} \leq y, Y_{i_{r+1}}>y,\ldots, Y_{i_{\nu-1}}>y ),
 $}
  \nonumber\\&
\end{align}
still with an inner summation  over all $r$-combinations $\{i_1, \ldots, i_r\}$ of the set $\{1, \ldots,\nu\}$.
\newline
By assumptions (a) and (b), Equation~\eqref{general1} becomes
\begin{equation}
F_{Y_{(k)}}(y)=\sum_{r=k}^\nu \sum_{_{1\leq i_1<\ldots<i_r\leq \nu}} \prod_{j=1}^r F_{Y_{i_j}}\left(q^{j-1}y\right)\prod_{m=r+1}^\nu \left(1- F_{Y_{i_m}}\left( q^{i_m-(m-r)}y \right) \right),
\end{equation}
where the inner summation is over all $r$-combinations $\{i_1, \ldots, i_r\}$ of the set $\{1, \ldots,\nu\}$.
\end{proof}
\pagebreak

In the next theorem, we assume that the non ordered random variables $Y_i, i=1,2,\ldots$ are dependent, $q$-uniformly distributed on the sets $[0,q^{i-1}t], t>0$, $ i=1,\ldots,\nu$, respectively, and we use the above lemma \ref{qdistrib}, to derive the $q$-distribution function and the $q$-density function of the corresponding maximum, minimum and $k$-th $q$-ordered random variables $Y_{(k)}, k=1,\ldots,\nu$. 
\begin{theorem}
\label{kordqunif} Let $Y_{1},\ldots,Y_{\nu}$ be dependent $q$-continuous random variables, $q$-uniformly distributed on the sets $[0,q^{i-1}t]$, $t>0$, $ i=1,\ldots,\nu$, respectively.
Assume that the random variables $Y_i$, $i=1,\ldots,\nu$, satisfy the dependence relations~\eqref{qdep1},~\eqref{qdep2},~\eqref{qdep3}. Then, for $y \in [0,t]$,
we have the following $q$-distribution functions and $q$-density functions:

\begin{itemize}
\item For the maximum $q$-ordered random variable $Y_{(\nu)}=\max\{Y_{1},\ldots,Y_{\nu}\}$ we have
\begin{equation}
\label{max2}
F_{Y_{(\nu)}}(y)=\frac{y^\nu}{t^\nu} 
\end{equation}
and
\begin{equation}
\label{max3}
f_{Y_{(\nu)}}(y)=[\nu]_q\frac{y^{\nu-1}}{t^\nu}.
\end{equation}
\item  For the $q$-distribution function and $q$-density function of the minimum $q$-ordered random variable $Y_{(1)}=\min\{Y_{1},\ldots,Y_{\nu}\}$ we have
\begin{equation}
\label{min2}
F_{Y_{(1)}}(y)=1-\prod_{i=1}^\nu \left(1-\frac{y}{q^{i-1}t }\right)
\end{equation}
and
\begin{equation}
\label{min3}
f_{Y_{(1)}}(y)=\frac{[\nu]_q}{q^{\nu-1}t}\prod_{i=1}^{\nu-1} \left(1-\frac{y}{q^{i-1}t }\right).
\end{equation}
\item 
For the $q$-distribution function and the $q$-density function of the $k$-th $q$-ordered random variable $Y_{(k)}$  we have
\begin{equation}
\label{general2}
F_{Y_{(k)}}(y)=\sum_{r=k}^\nu \qbinm{\nu}{r}\frac{y^r}{t^r}\prod_{i=1}^{\nu-r} \left(1-\frac{y}{q^{i-1}t} \right)
\end{equation}
and 
\begin{equation}
\label{general3}
f_{Y_{(k)}}(y)=\frac{[\nu]_q!q^{\binom{\nu-k}{2}}}{[k-1]_q![\nu-k]_q!q^{\binom{\nu}{2}-\binom{k}{2}}}\frac{y^{k-1}}{t^k} \prod_{j=1}^{\nu-k}\left(1-\frac{y}{q^{i-1}t} \right).
\end{equation}
\end{itemize}
\end{theorem}
\begin{proof}
The theorem assumptions allow us to use Equation~\eqref{max}; the $q$-distribution function of $Y_{(\nu)}$ is thus 
\begin{equation}
F_{Y_{(\nu)}}(y)
=\prod_{i=1}^\nu F_{Y_i}(q^{i-1}y)=
\frac{y}{t} \frac{qy}{qt}\cdots \frac{q^{\nu-1}}{q^{\nu-1}t}=\frac{y^\nu}{t^\nu}.
\end{equation}

\pagebreak

Taking the $q$-derivative of the above relation we have that the $q$-density function of $Y_{(\nu)}$ is straightforwardly given for $y \in [0,t]$ by 
\begin{equation}
\label{max3b}
f_{Y_{(\nu)}}(y)=d_q  F_{Y_{(\nu)}}(y)=[\nu]_q\frac{y^{\nu-1}}{t^\nu}.
\end{equation}
Note that 
\begin{equation*}\int_0^t f_{Y_{(\nu)}}(y) d_qy=\int_0^t [\nu]_q\frac{y^{\nu-1}}{t^\nu}d_qy=1,\end{equation*}
which is coherent with the fact we have here a $q$-density function.
\newline
Also, the $q$-distribution function of $Y_{(1)}$, by Equation~\eqref{min} of the previous lemma~\ref{qdistrib}, is straightforwardly given for $y \in [0,t]$ by
 \begin{equation}
F_{Y_{(1)}}(y)=
1-\prod_{i=1}^\nu \left(1-F_{Y_i}(y) \right)=1-\prod_{i=1}^\nu \left(1-\frac{y}{t} \right).
\end{equation}
Taking the $q$-derivative of the above relation and using the $q$-binomial formula~\eqref{qbinom1}, we have that the $q$-density function of $Y_{(\nu)}$ is expressed as
\begin{align}
\label{min3b}
f_{Y_{(1)}}(y)&= d_q  F_{Y_{(1)}}(y)=-\sum_{k=0}^\nu (-1)^k q^{- \binom{k}{2}} \qbinm{\nu}{k}\frac{[k]_qy^{k-1}}{t^k}
\nonumber\\&=\frac{[n]_q}{t}\sum_{k=0}^{\nu-1} (-1)^{k-1} q^{- \binom{k}{2}}q^{-k(\nu-k)}\qbin{\nu-1}{k-1} \frac{y^{k-1}}{t^{k-1}}
\nonumber\\&=\frac{[n]_q}{q^{\nu-1}t} \sum_{k=0}^{\nu-1} (-1)^{k-1} q^{- \binom{k-1}{2}}q^{-(k-1)(\nu-k)} \qbin{\nu-1}{k-1} \frac{y^{k-1}}{t^{k-1}}
\nonumber\\&=\frac{[n]_q}{q^{\nu-1}t} \sum_{k=0}^{\nu-1} (-1)^{k-1} q^{- \binom{k-1}{2}}\qbinm{\nu-1}{k-1}\frac{y^{k-1}}{t^{k-1}}
=\frac{[\nu]_q}{q^{\nu-1}t}\prod_{i=1}^{\nu-1} \left(1-\frac{y}{q^{i-1}t }\right).
\end{align}
Note that using the $q$-binomial formula~\eqref{qbinom1} and the $q$-identity~\eqref{q-identity1}, we obtain
\begin{align}
\label{densconf1}
\int_0^t f_{Y_{(1)}}(y) d_qy&=\frac{[\nu]_q}{q^{\nu-1}}\sum_{j=0}^{\nu-1}(-1)^j q^{-\binom{j}{2}}
\qbinm{\nu-1}{j}\int_0^t\frac{y^j}{t^{j+1}}d_qy\nonumber\\&=\frac{[\nu]_q}{q^{\nu-1}}\sum_{j=0}^{\nu-1}(-1)^jq^{-\binom{j}{2}}q^{-j(\nu-1-j)}
\qbin{\nu-1}{j}\frac{1}{[j+1]_q}\nonumber\\
&=[\nu]_q\sum_{j=0}^{\nu-1}(-1)^jq^{\binom{j+1}{2}}q^{-(j+1)(\nu-1)}\qbin{\nu-1}{j}\frac{1}{[j+1]_q}=[\nu]_q\frac{1}{\qbin{\nu}{\nu-1}
}=1,
\end{align}
which is coherent with the fact we have here a $q$-density function.
\pagebreak

Moreover, thanks to Equation~\eqref{general}, the $q$-distribution function of $Y_{(k)}$ is given by
\begin{align}
F_{Y_{(k)}}(y)&=
 \sum_{r=k}^\nu \sum_{_{1\leq i_1<\ldots<i_r\leq \nu}} \prod_{j=1}^r F_{Y_{i_j}}\left(q^{j-1}y\right)\prod_{m=r+1}^\nu \left(1- F_{Y_{i_m}}\left( q^{i_m-(m-r-1)}y \right) \right)\nonumber\\&=\sum_{r=k}^\nu \sum_{_{1\leq i_1<\ldots<i_r\leq \nu}}\frac{y}{q^{i_1-1}t}\frac{qy}{q^{i_2-1}t}\cdots\frac{q^{r-1}y}{q^{i_r-1}t}\left(1-\frac{y}{t}\right)\left(1-\frac{y}{qt}\right)\cdots\left(1-\frac{y}{q^{\nu-r-1}t}\right)
\nonumber\\
&=\sum_{r=k}^\nu\frac{y^r}{t^r}\prod_{i=1}^{\nu-r}\left(1-\frac{y}{q^{i-1}t}\right) \sum_{_{1\leq i_1<\ldots<i_r\leq \nu}}q^{-i_1-\cdots-i_r+\binom{r+1}{2}},
\end{align}
where the inner summation is over the $r$-combinations $\{i_1, \ldots, i_r\}$ of the set $\{1, \ldots,\nu\}$.

Applying the formula of the $q$-binomial coefficient~\eqref{combinations} in the above equation, we obtain
\begin{equation}
\label{general2b}
F_{Y_{(k)}}(y)=\sum_{r=k}^\nu \qbinm{\nu}{r}\frac{y^r}{t^r}\prod_{i=1}^{\nu-r} \left(1-\frac{y}{q^{i-1}t} \right), y \in [0,t].
\end{equation}
Taking the $q$-derivative of the above $q$-distribution function, using suitably the $q$-binomial formula~\eqref{qbinom1} and conducting all the needed algebraic manipulations, we have that the $q$-density function of $Y_{(k)}$ for $1\leq k \leq \nu$ is expressed (for $y \in [0,t]$) as
\begin{align} \label{general3b}
f_{Y_{(k)}}(y)&=d_q  F_{Y_{(k)}}(y)=\sum_{r=k}^\nu q^{-r(\nu-r)} \qbin{\nu}{r} [r]_q\frac{y^{r-1}}{t^r}\prod_{i=1}^{\nu-r} \left(1-\frac{y}{q^{i-1}t} \right)\nonumber\\
& + \sum_{r=k}^\nu q^{-r(\nu-r)}\qbin{\nu}{r}\frac{q^ry^{r}}{t^r}\sum_{j=0}^{\nu-r} (-1)^jq^{-\binom{j}{2}-j(\nu -r-j)} \qbin{\nu-r}{j} [j]_q \frac{y^{j-1}}{t^j}\nonumber\\
&= \frac{[\nu]_q}{t} \sum_{r=k}^\nu q^{-(\nu-r)} q^{-(r-1)(\nu-r)}\qbin{\nu-1}{r-1}\frac{y^{r-1}}{t^{r-1}}\prod_{i=1}^{\nu-r} \left(1-\frac{y}{q^{i-1}t} \right)\nonumber\\
&\phantom{=} -\frac{1}{t}\sum_{r=k}^\nu q^{-r(\nu-r)}\qbin{\nu}{r}\frac{q^ry^{r}}{t^r} [\nu-r]_q q^{-(\nu-r-1)}\nonumber\\
&\hspace{1.2cm} \times \sum_{j=0}^{\nu-r} (-1)^{j-1}q^{- \binom{j-1}{2}-(j-1)(\nu -r-j)}\qbin{\nu-r-1}{j-1} \frac{y^{j-1}}{t^{j-1}}\nonumber\\
&=\frac{[\nu]_q}{t} \sum_{r=k}^\nu q^{-(\nu-r)} q^{-(r-1)(\nu-r)}\qbin{\nu-1}{r-1} \frac{y^{r-1}}{t^{r-1}}\prod_{i=1}^{\nu-r} \left(1-\frac{y}{q^{i-1}t} \right)\nonumber\\
  &\phantom{=}-\frac{[\nu]_q}{t} \sum_{r=k}^{\nu -1} q^{-(\nu-r-1)}q^{-r(\nu-1-r)}\qbin{\nu-1}{r}\frac{y^{r}}{t^r} \prod_{i=1}^{\nu-r-1} \left(1-\frac{y}{q^{i-1}t} \right)\nonumber\\
&=\frac{[\nu]_q}{t} q^{-(\nu-k)} q^{-(k-1)(\nu-k)}\qbin{\nu-1}{k-1}\frac{y^{k-1}}{t^{k-1}}\prod_{i=1}^{\nu-k} \left(1-\frac{y}{q^{i-1}t} \right)\nonumber\\
&=q^{-k(\nu-k)} \frac{[\nu]_q!}{[k-1]_q![\nu-k]_q!}\frac{y^{k-1}}{t^k} \prod_{j=1}^{\nu-k}\left(1-\frac{y}{q^{i-1}t} \right)\nonumber\\
&=\frac{[\nu]_q!q^{\binom{\nu-k}{2}}}{[k-1]_q![\nu-k]_q!q^{\binom{\nu}{2}-\binom{k}{2}}}\frac{y^{k-1}}{t^k} \prod_{j=1}^{\nu-k}\left(1-\frac{y}{q^{i-1}t} \right).
\end{align}

\pagebreak

Note that using suitably the $q$-binomial formula~\eqref{qbinom1}, the $q$-identity~\eqref{q-identity1} and carrying out all the needed algebraic manipulations, we obtain
\begin{align}
\label{densconf2}
\int_0^t f_{Y_{(k)}}(y) d_qy&=\frac{[\nu]_q!q^{\binom{\nu-k}{2}}}{[k-1]_q![\nu-k]_q!q^{\binom{\nu}{2}-\binom{k}{2}}t}\int_0^t\frac{y^{k-1}}{t^{k-1}} \prod_{j=1}^{\nu-k}\left(1-\frac{y}{q^{i-1}t} \right)d_qy\nonumber\\
&=q^{-k(\nu-k)}\frac{[\nu]_q!}{[k-1]_q![\nu-k]_q!}\sum_{j=0}^{\nu-k}(-1)^j q^{-\binom{j}{2}}
\qbinm{\nu-k}{j} \int_0^t\frac{y^{k+j-1}}{t^{j+k}}d_qy\nonumber\\
&=q^{-k(\nu-k)}\frac{[\nu]_q!}{[k]_q![\nu-k]_q!}\sum_{j=0}^{\nu-k}(-1)^j q^{-\binom{j}{2}}q^{-j(\nu-k-j)} \qbin{\nu-k}{j}\frac{[k]_q}{[k+j]_q}\nonumber\\
&=\frac{[\nu]_q!}{[k]_q![\nu-k]_q!}\sum_{j=0}^{\nu-k}(-1)^j q^{\binom{j+1}{2}-(\nu-k)(k+j)}
\qbin{\nu-k}{j} \frac{[k]_q}{[k+j]_q}\nonumber\\
&=\frac{[\nu]_q!}{[k]_q![\nu-k]_q!}\frac{1}{\qbin{\nu}{\nu-k}}=1,
\end{align}
which is coherent with the fact we have here a $q$-density function.
\end{proof}
 \begin{remark}
 \label{qpowerlowqbetadistr}
 The random variables $Y_{(1)}$ and $Y_{(\nu)}$ follow $q$-power law  distributions
 (see Formulas~\eqref{max3} and~\eqref{min3})
 while the random variables $Y_{(2)},\dots,Y_{(\nu-1)}$
 follow $q$-beta  distributions (see~Formula~\eqref{general3}).
 \end{remark}
 In the following lemma, we consider the non-ordered 
 $q$-continuous random variables, $Y_{1},\ldots,Y_{\nu}$, being dependent and not identically distributed, and 
 we derive the joint $q$-distribution function of the
 $q$-ordered random variables, $Y_{(1)}$ and $Y_{(\nu)}$
 that satisfy inequalities~\eqref{qordered}.
\begin{lemma}
\label{qdistrib2}
 Let $Y_{1},\ldots,Y_{\nu}$ be dependent $q$-continuous random variables, where 
 \begin{itemize}
\item[(a)] Each $Y_i$ is defined on the set $R_{Y_i}$ from Formula~\eqref{sets}.
\item[(b)] Each $Y_i$ has a $q$-distribution function $F_{Y_i}(y)=P\left(Y_i\leq y\right)$, for $y \in R_{Y_i}$, of the same functional form and satisfy the dependence relations~\eqref{qdep1},~\eqref{qdep2},~\eqref{qdep3}.
\end{itemize}
Then,
 the joint $q$-distribution function of the $q$-ordered random variables 
 \begin{equation*}
 Y_{(1)}=\min\{Y_1,\ldots,Y_{\nu}\}  \text{\quad and \quad} Y_{(\nu)}=\max\{Y_{1},\ldots,Y_{\nu}\},
 \end{equation*} 
 is given by
\begin{equation}
\label{minmax}
F_{Y_{(1)},Y_{(\nu)}}(y,z)=\prod_{i=1}^\nu F_{Y_i}(q^{i-1}z)-\prod_{i=1}^\nu \left(F_{Y_i}(q^{i-1}z)-F_{Y_i}(y)\right)
\end{equation}
with $ y<q^{\nu-1}z,\nu\geq 1, y,z \in [0, \beta]$.
\end{lemma}
\pagebreak

\begin{proof}
Let $F_{Y_{(1)},Y_{(\nu)}}(y,z)$, $ y<q^{\nu-1}z,\nu\geq 1, y,z \in [0, \beta]$, be the joint $q$-distribution function of the random variables $Y_{(1)}$ and $Y_{(\nu)}$. 
Using the expression 
\begin{equation*}P\left(Y_{(1)}\leq y,Y_{(\nu)}\leq z\right)=P\left(Y_{(\nu)}\leq z \right)-P\left(Y_{(1)}>y,Y_{(\nu)}\leq z \right)\end{equation*}
we then have
\begin{align}
\label{minmax2}
&F_{Y_{(1)},Y_{(\nu)}}(y,z)=P\left(Y_{(1)}\leq y,Y_{(\nu)}\leq z\right)=P\left(Y_{(\nu)}\leq z \right)-P\left(Y_{(1)}>y,Y_{(\nu)}\leq z \right)\nonumber\\
&=P\left( Y_1 \leq z,Y_2 \leq z,\ldots,Y_{\nu} \leq z \right)-P\left( y <Y_{1} \leq z,y<Y_{2} \leq z,\ldots,y<Y_{\nu} \leq z \right)\nonumber\\
&=P\left( Y_1 \leq z) P(Y_2 \leq z \vert Y_1 \leq z) \cdots P(Y_{\nu} \leq z \vert Y_{1} \leq z, Y_{2} \leq z,\ldots, Y_{\nu-1} \leq z \right)\nonumber\\
&\phantom{=}-P( y<Y_{1} \leq z) P(y<Y_{2} \leq z \vert y<Y_{1} \leq z) \cdots \nonumber\\ 
&\phantom{=-}\cdot P(y<Y_{\nu} \leq z \vert y<Y_{1} \leq z, y<Y_{2} \leq z,\ldots, y<Y_{\nu-1} \leq z ).
\end{align}
By assumptions (a) and (b), Equation~\eqref{minmax2} becomes (for  $y,z \in [0, \beta]$ such that $y<q^{\nu-1}z$):
\begin{equation*}
F_{Y_{(1)},Y_{(\nu)}}(y,z)=\prod_{i=1}^\nu F_{Y_i}(q^{i-1}z)-\prod_{i=1}^\nu \left(F_{Y_i}(q^{i-1}z)-F_{Y_i}(y)\right).\qedhere
\end{equation*}
\end{proof}
In the next theorem, we assume that the non ordered random variables $Y_i$ are dependent and $q$-uniformly distributed on the sets $[0,q^{i-1}t]$ (for $t>0$), and we use the above lemma~\ref{qdistrib2}, to derive the joint $q$-distribution function and the joint $q$-density function of the
 $q$-ordered random variables $Y_{(1)}$ and $Y_{(\nu)}$.
\begin{theorem}
\label{minmaxqunif}
 Let $Y_{1},\ldots,Y_{\nu}$ be dependent $q$-continuous random variables, $q$-uniformly distributed on the sets $[0,q^{i-1}t], t>0$, $ i=1,\ldots,\nu$, respectively.
Assume that the random variables $Y_i$ satisfy the dependence relations~\eqref{qdep1},~\eqref{qdep2},~\eqref{qdep3}. Then,
 the joint $q$-distribution function and the joint $q$-density function of the $q$-ordered random variables, $Y_{(1)}=\min\{Y_1,\ldots,Y_{\nu}\}$ and $Y_{(\nu)}=\max\{Y_{1},\ldots,Y_{\nu}\}$ are given respectively by
 \begin{equation}
 \label{jointminmax}
 F_{Y_{(1)},Y_{(\nu)}}(y,z)=\frac{z^\nu}{t^\nu}-\frac{z^\nu}{t^\nu}\prod_{i=1}^\nu \left(1-\frac{y}{q^{i-1}z}\right)
\end{equation}
 and
 \begin{equation}
 \label{jointdensitminmax}
 f_{Y_{(1)},Y_{(\nu)}}(y,z)=q^{-\nu+1}[\nu]_q[\nu-1]_q\frac{z^{\nu-2}}{t^\nu}\prod_{i=1}^{\nu-2}\left(1-\frac{y}{q^{i}z}\right)
\end{equation}
with $ y<q^{\nu-1}z,\nu\geq 1, y,z \in [0,t]$.
 \end{theorem}
\begin{proof}With the conditions of the theorem, by~Equation~\eqref{minmax}, the $q$-distribution function of the random variables $Y_{(1)}$ and $Y_{(\nu)}$ becomes
 \begin{align}
\label{minmax3}
&F_{Y_{(1)},Y_{(\nu)}}(y,z)
=\prod_{i=1}^\nu F_{Y_i}(q^{i-1}z)-\prod_{i=1}^\nu \left(F_{Y_i}(q^{i-1}z)-F_{Y_i}(y)\right) \nonumber\\&\,\,\,\,\,=\frac{z}{t}\frac{qz}{qt}\cdots\frac{q^{\nu-1}z}{q^{\nu-1}t}-\left(\frac{z}{t}-\frac{y}{t}\right)\left(\frac{qz}{qt}-\frac{y}{qt}\right)\left(\frac{q^2z}{q^2t}-\frac{y}{q^2t}\right)\cdots\left(\frac{q^{\nu-1}z}{q^{\nu-1}t}-\frac{y}{q^{\nu-1}t}\right)\nonumber\\&\,\,\,\,\,=\frac{z^\nu}{t^\nu}-\frac{z^\nu}{t^\nu}\prod_{i=1}^\nu \left(1-\frac{y}{q^{i-1}z}\right).
\end{align}
\pagebreak

Taking the partial $q$-derivatives of the above joint $q$-distribution function and using the $q$-binomial formula~\eqref{qbinom1}, we have that the joint $q$-density function of the random variables $Y_{(1)}$ and $Y_{(\nu)}$ is expressed as
 \begin{align}
 \label{jointdensitminmax2}
 f_{Y_{(1)},Y_{(\nu)}}(y,z)&=\frac{\partial F_{Y_{(1)},Y_{(\nu)}}(y,z)}{\partial_qz \partial_qy}=-\frac{1}{\partial_qz \partial_qy}\frac{z^\nu}{t^\nu}\prod_{i=1}^\nu \left(1-\frac{y}{q^{i-1}z}\right)\nonumber\\&=-\frac{1}{\partial_qz \partial_qy}\frac{z^\nu}{t^\nu}\sum_{r=0}^\nu(-1)^rq^{-\binom{r}{2}}\qbinm{\nu}{r}\frac{y^r}{z^r}\nonumber\\&=-\frac{1}{t^\nu}\sum_{r=0}^\nu(-1)^{r}q^{-\binom{r}{2}}q^{-r(\nu-r)}\qbin{\nu}{r}[r]_q[\nu-r]_q y^{r-1}z^{\nu-r-1}\nonumber\\&=\frac{z^{\nu-2}}{t^\nu}[\nu]_q[\nu-1]_q\sum_{r=0}^\nu(-1)^{r-1}q^{-\binom{r}{2}}q^{-r(\nu-r)}\qbin{\nu-2}{r-1}\frac{ y^{r-1}}{z^{r-1}}
 \nonumber\\&=\frac{z^{\nu-2}}{q^{\nu-1}t^\nu}[\nu]_q[\nu-1]_q\sum_{r=0}^\nu(-1)^{r-1}q^{-\binom{r-1}{2}}q^{-(r-1)(\nu-r-1)}\qbin{\nu-2}{r-1}\frac{ y^{r-1}}{(qz)^{r-1}} \nonumber\\&=\frac{z^{\nu-2}}{q^{\nu-1}t^\nu}[\nu]_q[\nu-1]_q\sum_{j=0}^\nu(-1)^{j}q^{-\binom{j}{2}}q^{-j(\nu-2-j)}\qbin{\nu-2}{j}\frac{ y^{j}}{(qz)^{j}}\nonumber\\&=q^{-\nu+1}[\nu]_q[\nu-1]_q\frac{z^{\nu-2}}{t^\nu}\prod_{i=1}^{\nu-2}\left(1-\frac{y}{q^{i}z}\right).
\end{align}

Note that using suitably the $q$-binomial formula~\eqref{qbinom1}, the $q^{-1}$-identity~\eqref{q-identity2} and carrying out all the needed algebraic manipulations, we obtain
\begin{align}
\label{jointdensminmax3}
& \int_0^t \int_0^{q^{\nu-1}z}f_{Y_{(1)},Y_{(\nu)}}(y,z)d_qy d_qz\nonumber\\&=\frac{q^{-\nu+1}}{t^\nu}[\nu]_q[\nu-1]_q
\sum_{j=0}^\nu(-1)^{j}q^{-\binom{j}{2}}q^{-j(\nu-2-j)-j}\qbin{\nu-2}{j} \int_0^t \int_0^{q^{\nu-1}z}y^jz^{\nu-2-j}d_qy d_qz\nonumber\\&=\frac{q^{-\nu+1}}{t^\nu}[\nu]_q[\nu-1]_q \sum_{j=0}^\nu(-1)^{j}q^{-\binom{j}{2}}q^{-j(\nu-2-j)-j}\qbin{\nu-2}{j}\frac{q^{(\nu-1)(j+1)}t^\nu}{[j+1]_q[\nu]_q}\nonumber\\&=q^{-\nu+2}[\nu-1]_q \sum_{j=0}^\nu(-1)^{j}q^{\binom{j+1}{2}+\nu-2}\qbin{\nu-2}{j}\frac{[1]_q}{[1+j]_q}\nonumber\\&=q^{-\nu+2}[\nu-1]_q\frac{1}{\qbinm{\nu-1}{\nu-2}}=1
 \end{align}
 which is coherent with the fact we have here a $q$-density function.
\end{proof}
In the following lemma, we consider the non-ordered 
 $q$-continuous random variables, $Y_{1},\ldots,Y_{\nu}$, being dependent and not identically distributed, and 
 we derive the joint $q$-distribution function of the
 $q$-ordered random variables, $Y_{(k)}$ and $Y_{(r)}, \,1\leq k <r\leq \nu$. 
 
 \pagebreak
 
\begin{lemma}
\label{qdistribjointkr} 
 Let $Y_{1},\ldots,Y_{\nu}$ be dependent $q$-continuous random variables, where 
 \begin{itemize}
\item[(a)] Each $Y_i$ is defined on the set $R_{Y_i}$ from Formula~\eqref{sets}.
\item[(b)] Each $Y_i$ has a $q$-distribution function $F_{Y_i}(y)=P\left(Y_i\leq y\right), y \in R_{Y_i}$ of the same functional form and satisfies the dependence relations~\eqref{qdep1},~\eqref{qdep2},~\eqref{qdep3}.
\end{itemize}
Then,
 the joint $q$-distribution function of the $q$-ordered random variables $Y_{(k)}$ and $Y_{(r)}$ for $1\leq k <r\leq \nu$, 
 where $Y_{(i)}$, $i=1,\dots,\nu$, satisfy inequalities~\eqref{qordered}, 
 is given by
\begin{align}
\label{jointkr}
&F_{Y_{(k)},Y_{(r)}}(y,z)\nonumber\\
&=\sum_{j=r}^\nu \sum_{s=k}^j \sum \prod_{n_1=1}^s F_{Y_{i_{n_1}}}\left(q^{n_1-1}y\right)
\prod_{ni_1, \ldots, i_r=s+1}^j \left( F_{Y_{i_{n_2}}}\left(q^{n_2-s-1}z\right)- F_{Y_{i_{n_2}}}\left(y\right)\right)\nonumber\\&\,\,\,\,\cdot
\prod_{n_3=j+1}^\nu \left(1- F_{Y_{i_{n_3}}}\left( q^{i_{n_3}-(n_3-j)}z \right) \right), y<q^{r-k}z, 1\leq k <r\leq \nu, y,z \in [0,\beta],
\end{align}
where  the inner summation is over all pairwise disjoint subsets $\{i_1,\ldots,i_s \}$ and $\{i_{s+1},\ldots,i_{j} \}$ of the set $\{1,\ldots,\nu\}$ with
$1\leq i_1<\cdots<i_s \leq \nu$ and $1\leq i_{s+1}< i_{s+2}<\cdots<i_{j} \leq \nu$.
\end{lemma}
\begin{proof}
Let $F_{Y_{(k)},Y_{(r)}}(y,z)=P\left(Y_{(k)}\leq y,Y_{(r)}\leq z\right)$, $
y<q^{r-k}z, 1\leq k <r\leq \nu, y,z \in [0, \beta]$, be the joint $q$-distribution function of the random variables $Y_{(k)}$ and $Y_{(r)}$ with $1\leq k <r\leq \nu$. 
Then, the events $Y_{(k)}\leq y$ and $Y_{(r)}\leq z$ occur if and only if at least $k$ random variables in $\{Y_1,\ldots,Y_\nu\}$ take values in the set $[0, y]$, 
while $r-k$ random other variables take values in the set  $(y,z]$, and the remaining ones take values in the set $(z, \beta]$, $1\leq k <r\leq \nu$.
So, for $
y<q^{r-k}z, 1\leq k <r\leq \nu, y,z \in [0, \beta]$, we have
\begin{align}
\label{jointkr2b}
&F_{Y_{(k)},Y_{(r)}}(y,z)=P\left(Y_{(k)}\leq y,Y_{(r)}\leq z\right) \nonumber\\
&=\sum_{j=r}^\nu \sum_{s=k}^j
\sum\limits_{\substack{ 1\leq i_1<\ldots<i_s \leq \nu \\ 1\leq i_{s+1}< i_{s+2}<\ldots<i_j \leq \nu  }}
 P\big( \{Y_{i_\ell}\leq y\}_{\ell=1,\dots,s}, \{ y<Y_{i_\ell}\leq z\}_{\ell=s+1,\dots,j},   \{Y_{i_\ell}> z\}_{\ell = j+1,\dots,\nu}\big)
\nonumber\\
&=
\sum_{j=r}^\nu \sum_{s=k}^j
\sum\limits_{\substack{ 1\leq i_1<\ldots<i_s \leq \nu \\ 1\leq i_{s+1}< i_{s+2}<\ldots<i_j \leq \nu  }} 
 \nonumber\\&
P(Y_{i_1} \leq y) P(Y_{i_2} \leq y \vert Y_{i_1} \leq y) \cdots P(Y_{i_s} \leq y \vert Y_{i_1} \leq y, Y_{i_2} \leq y,\ldots, Y_{i_{s-1}} \leq y)\nonumber\\&
\cdot P\big(y<Y_{i_{s+1}}\leq z \vert Y_{i_1} \leq y, 
\ldots, Y_{i_{s}} \leq y \big)
\nonumber\\&\cdot P\big(y<Y_{i_{s+2}}\leq z \vert Y_{i_1} \leq y, 
\ldots, Y_{i_{s}} \leq y,y<Y_{i_{s+1}}\leq z \big)\cdots
\nonumber\\
&\cdot P\big(y<Y_{i_{j}}\leq z \vert Y_{i_1} \leq y,
\ldots, Y_{i_{s}} \leq y,y<Y_{i_{s+1}}\leq z,\ldots,y<Y_{i_{j-1}}\leq z \big)\nonumber\\
&\cdot P\big(Y_{i_{j+1}}> z \vert Y_{i_1} \leq y,
\ldots, Y_{i_{s}} \leq y,y<Y_{i_{s+1}}\leq z,\ldots,y<Y_{i_{j}}\leq z \big)\cdots
\nonumber\\
& \cdot 
P\big(Y_{i_{\nu}}> z \vert Y_{i_1} \leq y, \ldots, Y_{i_{s}} \leq y,y<Y_{i_{s+1}}\leq z,\ldots,y<Y_{i_{j}}\leq z,Y_{i_{j+1}}> z,\ldots,Y_{i_{\nu-1}}> z \big),\nonumber\\
&
\end{align}
where the inner summation is over all pairwise disjoint subsets $\{i_1, \ldots,i_s \}$ and $\{i_{s+1}, \ldots,i_{j} \}$ of the set $\{1,\ldots,\nu\}$ with
$1\leq i_1<\ldots<i_s \leq \nu$ and $1\leq i_{s+1}< i_{s+2}<\ldots<i_{j} \leq \nu$.
\pagebreak

By assumptions (a) and (b), Equation~\eqref{jointkr2b} becomes
(for $y<q^{r-k}z$, $1\leq k <r\leq \nu$, and $y,z \in [0, \beta]$)
\begin{align}
\label{jointkr22}
F_{Y_{(k)},Y_{(r)}}(y,z)
&=\sum_{j=r}^\nu \sum_{s=k}^j 
\sum\limits_{\substack{1\leq i_1<\ldots<i_s \leq \nu \\ 1\leq i_{s+1}< i_{s+2}<\ldots<i_j \leq \nu}} \prod_{n_1=1}^s F_{Y_{i_{n_1}}}\left(q^{n_1-1}y\right)
\nonumber\\
&\phantom{=} \cdot 
\prod_{n_2=s+1}^j \left( F_{Y_{i_{n_2}}}\left(q^{n_2-s-1}z\right)- F_{Y_{i_{n_2}}}\left(y\right)\right)
\prod_{n_3=j+1}^\nu \left(1- F_{Y_{i_{n_3}}}\left( q^{i_{n_3}-(n_3-j)}z \right) \right),\nonumber
\end{align}
where the inner summation is over all pairwise disjoint subsets $\{i_1, \ldots,i_s \}$ and $\{i_{s+1},\ldots,i_{j} \}$ of the set $\{1,\ldots,\nu\}$ with
$1\leq i_1<\cdots<i_s \leq \nu$ and $1\leq i_{s+1}< i_{s+2}<\cdots<i_{j} \leq \nu$.
\end{proof}
In the next theorem, we use the above lemma~\ref{qdistribjointkr} to derive the joint $q$-distribution function and the joint $q$-density function
of the ordered random variables.
\begin{theorem}
\label{jointkrqunif}
 Let $Y_{1},\ldots,Y_{\nu}$ be dependent $q$-continuous random variables, $q$-uniformly distributed on the sets $[0,q^{i-1}t], t>0$, $ i=1,\ldots,\nu$, respectively. Assume that the random variables $Y_i$, $i=1,\ldots,\nu$, satisfy the dependence relations~\eqref{qdep1},~\eqref{qdep2},~\eqref{qdep3}. Then,
 the joint $q$-distribution function and the joint $q$-density function of the $q$-ordered random variables, $Y_{(k)}$ and $Y_{(r)}$, for $1 \leq k<r \leq \nu$, are given respectively by
 \begin{equation}
 \label{jointkrdistrunif}
 F_{Y_{(k)},Y_{(r)}}(y,z)=\sum_{j=r}^\nu \sum_{s=k}^j \qbinm{\nu}{s,j-s} \frac{y^s}{t^s}\frac{z^{j-s}}{t^{j-s}}\prod_{i=1}^{j-s} \left(1-\frac{y}{q^{i-1}z}\right)\prod_{m=1}^{\nu-j} \left(1-\frac{z}{q^{m-1}t}\right)
\end{equation}
 and
 \begin{equation}
 \label{jointdensitkrunif}
 \resizebox{1.00\hsize}{!}{$
f_{Y_{(k)},Y_{(r)}}(y,z)=\frac{q^{-r(\nu-r)}q^{-k(r-k)}[\nu]_q!}{[k-1]_q![r-k-1]_q![\nu-r]_q!}\frac{y^{k-1}}{t^r}z^{r-k-1}\prod_{i=1}^{r-k-1}\left(1-\frac{y}{q^i z}\right)\prod_{m=1}^{\nu-r}\left(1-\frac{z}{q^{m-1}t}\right)
$}
\end{equation}
with $ y<q^{r-k}z, 1 \leq k<r \leq \nu, y,z \in [0,t]$.
 \end{theorem}
 \begin{proof}
By Equation~\eqref{jointkr} and the $q$-multinomial formulas \eqref{qmultinomialseries} and~\eqref{qmultinomial2}, 
the joint $q$-distribution function of $Y_{(k)}$ and $Y_{(r)}$ satisfies 
 \begin{align}
 \label{jointkr2}
& F_{Y_{(k)},Y_{(r)}}(y,z)
=\sum_{j=r}^\nu \sum_{s=k}^j 
\sum\limits_{\substack{ 1\leq i_1<\cdots<i_s \leq \nu \\ 1\leq m_{1}< m_{2}<\cdots<m_{j-s} \leq \nu }}q^{\binom{s+1}{2}}q^{\binom{j-s+1}{2}}q^{-i_1-\cdots-i_s}q^{-m_1-\cdots-m_{j-s}}
\nonumber\\&\,\,\,\,\,\,\cdot \frac{y^s}{t^s}\frac{z^{j-s}}{t^{j-s}}\left(1-\frac{y}{z}\right)\left(1-\frac{y}{qz}\right)\cdots \left(1-\frac{y}{q^{j-s-1}z}\right)\left(1-\frac{z}{t}\right)\left(1-\frac{z}{qt}\right)\cdots \left(1-\frac{z}{q^{\nu-j-1}t}\right)\nonumber\\
&=\sum_{j=r}^\nu \sum_{s=k}^j \qbinm{\nu}{s,j-s} \frac{y^s}{t^s}\frac{z^{j-s}}{t^{j-s}}\prod_{i=1}^{j-s} \left(1-\frac{y}{q^{i-1}z}\right)\prod_{m=1}^{\nu-j} \left(1-\frac{z}{q^{m-1}t}\right),
\end{align}
where the inner summation of the first equality, is over all pairwise disjoint subsets $\{i_1, i_2,\ldots,i_s \}$ and $\{m_{1}, m_{2},\ldots,m_{j-s} \}$ of the set $\{1,\ldots,\nu\}$ with
$1\leq i_1<\ldots<i_s \leq \nu$ and $1\leq m_{1}< m_{2}<\ldots<m_{j-s} \leq \nu$. 
\pagebreak

 The above joint $q$-distribution~\eqref{jointkr2}, of the random variables $Y_{(k)}$ and $Y_{(r)}$, for $1 \leq k<r \leq \nu$ and $y<q^{r-k}z$,
 $ y,z \in [0,t]$ can be written as 
 \begin{equation*}
 F_{Y_{(k)},Y_{(r)}}(y,z)=\sum_{j=r}^\nu \frac{q^{-j(\nu-j)}}{t^j} \qbin{\nu}{j}\prod_{m=1}^{\nu-j} \left(1-\frac{z}{q^{m-1}t}\right)
\sum_{s=k}^j \qbin{j}{s} \left(\frac{z}{q^s}\right)^{j-s}\prod_{i=1}^{j-s} \left(1-\frac{y}{q^{i-1}z}\right). \nonumber
\end{equation*} 
Taking the partial $q$-derivative of the inner sum over $y$, using suitably the $q$-binomial formula~\eqref{qbinom1} and carrying out all needed algebraic manipulations, we obtain
\begin{align*}
& \frac{\partial_q}{\partial_q y}\sum_{s=k}^j q^{-s(j-s)}\qbin{j}{s} z^{j-s}\prod_{i=1}^{j-s} \left(1-\frac{y}{q^{i-1}z}\right)\\
&=[j]_q \left(\sum_{s=k}^jq^{-s(j-s)}\qbin{j-1}{s-1}y^{s-1}z^{j-s} \prod_{i=1}^{j-s}\left(1-\frac{y}{q^{i-1}z} \right)\right)
 -[j]_q \left(\sum_{s=k}^j q^{-(s+1)(j-s-1)}\qbin{j-1}{s}  \right.\\
& \hspace{2cm} \left.  y^s  z^{j-s-1} \prod_{i=1}^{j-s-1}\left(1-\frac{y}{q^{i-1}z}\right)
q^{-k(j-k)}[j]_q\qbin{j-1}{k-1}y^{k-1}z^{j-k} \prod_{i=1}^{j-k} \left(1-\frac{y}{q^{i-1}z}\right)\right).
\end{align*}
So,
\begin{align*}
&\frac{\partial_q F_{Y_{(k)},Y_{(r)}}(y,z)}{\partial_q y }\\
&=\sum_{j=r}^\nu q^{-j(\nu-j)}\qbin{\nu}{j}\frac{1}{t^j}\prod_{m=1}^{\nu-j} \left(1-\frac{z}{q^{m-1}t}\right)q^{-k(j-k)}[j]_q\qbin{j-1}{k-1}y^{k-1}z^{j-k} \prod_{i=1}^{j-k} \left(1-\frac{y}{q^{i-1}z}\right)\\&
=\sum_{j=r}^\nu q^{-j(\nu-j)}q^{-k(j-k)}\qbin{\nu}{j}\qbin{j-1}{k-1}[j]_q\frac{1}{t^j}y^{k-1}z^{j-k}\prod_{i=1}^{j-k} \left(1-\frac{y}{q^{i-1}z}\right)\prod_{m=1}^{\nu-j} \left(1-\frac{z}{q^{m-1}t}\right)\\&
=\frac{[\nu]_q! y^{k-1}}{[k-1]_q![\nu-k]_q!}\sum_{j=r}^\nu q^{-j(\nu-j)}q^{-k(j-k)}\qbin{\nu-k}{j-k}\frac{z^{j-k}}{t^j}\prod_{i=1}^{j-k} \left(1-\frac{y}{q^{i-1}z}\right)
\prod_{m=1}^{\nu-j} \left(1-\frac{z}{q^{m-1}t}\right).
\end{align*}
In the last sum of this equation, taking the partial $q$-derivative over $z$, and using suitably $q$-binomial formula~\eqref{qbinom1}, we get
\begin{align}
\label{qpartialy3}
&\frac{\partial_q}{\partial_q z}\sum_{j=r}^\nu q^{-j(\nu-j)}q^{-k(j-k)}\qbin{\nu-k}{j-k}\frac{1}{t^j}z^{j-k}\prod_{i=1}^{j-k} \left(1-\frac{y}{q^{i-1}z}\right)\prod_{m=1}^{\nu-j} \left(1-\frac{z}{q^{m-1}t}\right)\nonumber\\&=
\sum_{j=r}^\nu q^{-j(\nu-j)}q^{-k(j-k)}\qbin{\nu-k}{j-k}[j-k]_q\frac{1}{t^j}z^{j-k-1}\prod_{i=1}^{j-k-1} \left(1-\frac{y}{q^{i}z}\right)\prod_{m=1}^{\nu-j} \left(1-\frac{z}{q^{m-1}t}\right)
\nonumber\\
 &\ \ -\sum_{j=r}^\nu q^{-j(\nu-j)-k(j-k)-(\nu-j-1)+j-k}\qbin{\nu-k}{j-k}[\nu-j]_q\frac{z^{j-k}}{t^{j+1}}\prod_{i=1}^{j-k} \left(1-\frac{y}{q^{i}z}\right)
 \prod_{m=1}^{\nu-j-1} \left(1-\frac{z}{q^{m-1}t}\right).
\end{align}

\pagebreak

\noindent It follows that
\begin{align}
&\frac{\partial_q}{\partial_q z}\sum_{j=r}^\nu q^{-j(\nu-j)}q^{-k(j-k)}\qbin{\nu-k}{j-k}\frac{1}{t^j}z^{j-k}\prod_{i=1}^{j-k} \left(1-\frac{y}{q^{i-1}z}\right)\prod_{m=1}^{\nu-j} \left(1-\frac{z}{q^{m-1}t}\right)\nonumber\\
&=[\nu-k]_q\sum_{j=r}^\nu q^{-j(\nu-j)}q^{-k(j-k)}\qbin{\nu-k-1}{j-k-1}\frac{1}{t^j}z^{j-k-1}\prod_{i=1}^{j-k-1} \left(1-\frac{y}{q^{i}z}\right)\prod_{m=1}^{\nu-j} \left(1-\frac{z}{q^{m-1}t}\right)\nonumber\\
&\phantom{=}\smallminus  [\nu \smallminus k]_q\sum_{j=r}^\nu q^{-j(\nu-j)-k(j-k)-(\nu-j-1)+{j-k}}\qbin{\nu \smallminus k\smallminus 1}{j \smallminus k}\frac{z^{j-k}}{t^{j+1}}\prod_{i=1}^{j-k}\! \left(1 \smallminus \frac{y}{q^{i}z}\right)\prod_{m=1}^{\nu-j-1} \!\left(1 \smallminus \frac{z}{q^{m-1}t}\right)\nonumber\\
&=q^{-r(\nu-r)}q^{-k(r-k)}[\nu-k]_q\qbin{\nu-k-1}{r-k-1}\frac{1}{t^r}z^{r-k-1}\prod_{i=1}^{r-k-1} \left(1-\frac{y}{q^{i}z}\right)\prod_{m=r}^{\nu-j} \left(1-\frac{z}{q^{m-1}t}\right).\nonumber
\end{align}

From this identity, we get that the joint $q$-density function given by
\begin{align*}
 f_{Y_{(k)},Y_{(r)}}(y,z)&=\frac{\partial_q^2 F_{Y_{(k)},Y_{(r)}}(y,z)}{\partial_q z\partial_q y }\\
&=
 \frac{q^{-r(\nu-r)}q^{-k(r-k)}[\nu]_q!}{[k-1]_q![r-k-1]_q![\nu-r]_q!}\frac{y^{k-1}}{t^r}z^{r-k-1}\prod_{i=1}^{r-k-1}\! \left(1\smallminus \frac{y}{q^{i}z}\right)\prod_{m=r}^{\nu-j} \left(1 \smallminus\frac{z}{q^{m-1}t}\right) 
\end{align*}
with $y<q^{r-k}z$, $y,z \in [0,t]$.
 \newline
Note that using suitably the $q$-binomial formula~\eqref{qbinom1}, the $q^{-1}$ and $q$-identities~\eqref{q-identity1},~\eqref{q-identity2} and carrying out all the needed algebraic manipulations, we obtain
\begin{align*}
&\int_0^t \int_0^{q^{r-k}z} f_{Y_{(k)},Y_{(r)}}(y,z)d_qy \, d_qz\\
& =\frac{q^{-r(\nu-r)}q^{-k(r-k)}[\nu]_q!t^{-r}}{[k- 1]_q![r-k-1]_q![\nu- r]_q!}\\
& \quad \times
 \int_0^t \left( \int_0^{q^{r-k}z}y^{k-1}\prod_{i=1}^{r-k-1} \left(1- \frac{y}{q^{i}z}\right)d_qy\right) z^{r-k-1} \prod_{m=r}^{\nu-j} \left(1-\frac{z}{q^{m-1}t}\right)d_qz\\
 &=\frac{q^{-r(\nu-r)}q^{-k(r-k)}[\nu]_q!t^{-r}}{[k-1]_q![r-k-1]_q![\nu-r]_q!}\frac{q^k}{[k]_q}\\
&  \quad \times \sum_{m=0}^{r-k-1}(-1)^m q^{\binom{m+1}{2}+(r-k-1)k}\qbin{r-k-1}{m}\frac{[k]_q}{[k+m]_q} \int_0^t z^{r-1} \prod_{m=r}^{\nu-j} \left(1-\frac{z}{q^{m-1}t}\right)d_qz\\
 & =\frac{q^{-k(r-k)}[\nu]_q!}{[k\smallminus 1]_q![r\smallminus k\smallminus 1]_q![\nu\smallminus r]_q!}\frac{1}{\qbinm{r \smallminus 1}{r\smallminus k\smallminus 1}} 
 \frac{q^k}{[k]_q [r]_q}\sum_{i=0}^{\nu-r}({\scriptscriptstyle{-}}1)^i q^{\binom{i+1}{2}-(i+r)(\nu-r)}\qbin{\nu \smallminus r}{i}\frac{[r]_q}{[r+i]_q}\\
&=\frac{[\nu]_q!}{[k]_q![r-k-1]_q![\nu-r]_q!}\frac{[r-k-1]_q![k]_q!}{[r]_q!} \frac{1}{\qbin{\nu}{\nu-r}}=1,
\end{align*}
which is coherent with the fact we have here a $q$-density function.
Note also that the joint $q$-distribution function and $q$-density function of the random variables $Y_{(1)}$ and $Y_{(\nu)}$ are given respectively by~\eqref{jointkrdistrunif} and~\eqref{jointdensitkrunif}, for $k=1$ and $r=\nu$.
 \end{proof}
 \begin{remark}
 \label{qdirichletdistr}
The bivariate random variables $\left( Y_{(k)},Y_{(r)}\right)$ for $1\leq k <r\leq \nu$ with joint $q$-density function~\eqref{jointdensitkrunif}, follow $q$-Dirichlet  distributions.
 \end{remark}
 
 In the following proposition, we consider the non-ordered 
 $q$-continuous random variables, $Y_{1},\ldots,Y_{\nu}$, being dependent and not identically distributed and 
 we derive the joint distribution function of the 
 $q$-ordered random variables, $Y_{(1)},\ldots,Y_{(\nu)}$.

 \begin{proposition}
 \label{jointdensityorder}
Let $(Y_{1},\ldots,Y_{\nu})$ be a $q$-continuous $\nu$-variate random vector with joint $q$-density function $f(y_1,\dots,y_{\nu})$.
 Then the $q$-density function of the $q$-ordered
random vector ${\mathcal{Y}}=(Y_{(1)},\ldots,Y_{(\nu)})$ is given by
 \begin{align}
 \label{multjointorder}
& f_{\mathcal{Y}} (y_{(1)},\ldots,y_{(\nu)})=\sum
f_{Y_{i_\nu}}(y_{(\nu)}) \ f_{Y_{i_{\nu-1}}|Y_{i_\nu}}(y_{{(\nu-1)}}|y_{(\nu)}) \ \cdots \ f_{Y_{i_1}|(Y_{i_2},\ldots,Y_{i_\nu})}(y_{(1)}|y_{(2)},\ldots,y_{(\nu)}),\nonumber\\&\,\,\,\,\,\,\,0<y_{(1)}<qy_{(2)}<y_{(2)}<qy_{(3)}<\cdots<y_{({\nu-1})}<qy_{(\nu)}<y_{(\nu)}<\beta, 
 \end{align}
 where the summation is over all permutations $(i_1,\ldots,i_\nu )$ of $\{1,\ldots,\nu\}$.
 \end{proposition}
 \begin{proof}
  The joint $q$-density function  is
 \bgroup \thinmuskip=.93\thinmuskip  \medmuskip=.93\medmuskip  \thickmuskip=.91\thickmuskip
 \begin{align}
 \label{multjointorder2}
& f_{\mathcal{Y}} (y_{(1)},\ldots,y_{(\nu)})=\frac{P\left(qy_{(1)}<Y_{(1)}\leq y_{(1)},\ldots,q y_{(\nu)}<Y_{(\nu)} \leq y_{(\nu)} \right)}{(1-q)y_{(1)}(1-q)y_{(2)}\cdots (1-q) y_{(\nu)}}\nonumber\\
 &=(1-q)^{-\nu}\prod_{i=1}^\nu y_{(i)}^{-1}\sum P\left(qy_{(1)}<Y_{i_1}\leq y_{(1)},\ldots,q y_{(\nu)}<Y_{i_\nu} \leq y_{(\nu)} \right)\nonumber\\&=
 (1-q)^{-\nu}\prod_{i=1}^\nu y_{(i)}^{-1}\sum P\left(qy_{(\nu)}<Y_{i_\nu}\leq y_{(\nu)} \right) P\left( qy_{(\nu-1)}<Y_{i_{\nu-1}} \leq y_{(\nu-1)}|qy_{(\nu)}<Y_{i_\nu}\leq y_{(\nu)} \right)\nonumber\\
 &\qquad \quad \qquad \qquad \qquad \cdots P\left(q y_{(\nu)}<Y_{i_1} \leq y_{(1)}|qy_{(1)}<Y_{i_1}\leq y_{(1)},\dots, qy_{(\nu)}<Y_{i_\nu}\leq y_{(\nu)} \right),
\nonumber\\
 \end{align}\egroup
 where the summation is over all permutations $(i_1,\ldots,i_\nu )$ of $\{1,\ldots,\nu\}$.
 \newline
 Applying Definition~\ref{dependentqjointdef}  on the dependent $q$-density function and the relations~\eqref{dependentqcont},~\eqref{multdependqdens}, to the above equation~\eqref{multjointorder2}, we obtain \ref{multjointorder}.
 \end{proof}
 
Next, we assume that the non ordered random variables $Y_i, i=1,\ldots,\nu$ are dependent and $q$-uniformly distributed on the sets $[0,q^{i-1}t], t>0$, $ i=1,\ldots,\nu$, respectively, 
and the joint $q$-density function of the 
 $q$-ordered random variables $Y_{(1)},\ldots,Y_{(\nu)}$, is obtained in the following corollary of Proposition~\ref{jointdensityorder}.
 
 \begin{corollary}
 \label{mutlquniform2}
 Let $Y_1,\ldots,Y_{\nu}$ be dependent $q$-continuous random variables, $q$-uniformly distributed on the sets $[0,q^{i-1}t], t>0$ $ i=1,\ldots,\nu$, respectively. Assume that the random variables $Y_i$, $i=1,\ldots,\nu$, satisfy the dependence relations~\eqref{qdep1},~\eqref{qdep2},~\eqref{qdep3}. Then the joint $q$-density function of the $\nu$-variate $q$-continuous random vector ${\mathcal{Y}}=(Y_{(1)},\ldots,Y_{(\nu)})$ with $Y_{(k)}$, $k=1,\ldots,\nu$, the $k$-th $q$-ordered random variables, is given by
 \begin{equation}
 \label{multjointorderunif}
 f_{\mathcal{Y}} (y_1,\ldots,y_\nu)=\frac{[\nu]_q!}{q^{\binom{\nu}{2}}t^\nu},\,\,0<y_1<qy_2<y_2<qy_3<\cdots<y_{\nu-1}<qy_\nu<y_\nu<t. 
 \end{equation}
 \end{corollary}
 \pagebreak
 \begin{proof}
 Let $Y_{1},\ldots,Y_{\nu}$ be dependent $q$-continuous random variables, with each $Y_i$ (for $ i=1,\ldots,\nu$) $q$-uniformly distributed on the set $[0,q^{i-1}t]$ (for some $t>0$). Applying~\eqref{multjointorder} of the previous proposition~\ref{jointdensityorder},  
 the joint $q$-density function of the $\nu$-variate $q$-continuous random vector ${\mathcal{Y}}=(Y_{(1)},\ldots,Y_{(\nu)})$ 
 (where each $Y_{(k)}$, for  $k=1,\ldots,\nu$, is  the $k$-th $q$-ordered random variable) is given (for $0<y_1<qy_2<y_2<qy_3<\cdots<y_{\nu-1}<qy_\nu<y_\nu<t$)
by
\begin{align}
 \label{multjointorderunif2}
 f_{\mathcal{Y}} (y_1,\ldots,y_\nu)&=\sum
\frac{1}{q^{i_\nu-1}t}\frac{1}{q^{i_{\nu-1}-1}t}\ldots \frac{1}{q^{i_{1}-1}t}\\
&=\frac{1}{t^\nu}\sum \frac{1}{\prod_i^\nu q^{i_j-1}},
\end{align}
where the summation is over all permutations $(i_1,\ldots,i_\nu )$ of $\{1,\ldots,\nu\}$.

So, 
\begin{equation}
 \label{multjointorderunif3}
 f_{\mathcal{Y}} (y_1,\ldots,y_\nu)=\frac{[\nu]_{q^{-1}}!}{t^\nu}=\frac{[\nu]_q!}{q^{\binom{\nu}{2}} t^\nu}.
\end{equation}
Note that
\begin{align}
\label{multjointorderunifjust}
&\int_{0}^t \int_{0}^{qy_\nu} \int_{0}^{qy_{\nu-1}}\ \cdots \int_{0}^{qy_3}\int_{0}^{qy_2} 
 f_{\mathcal{Y}} (y_1,\ldots,y_\nu)
d_qy_1\,d_qy_2 \cdots d_qy_{\nu-2}\, d_qy_{\nu-1}\,d_qy_\nu\nonumber\\&=
\int_{0}^t \int_{0}^{qy_\nu} \int_{0}^{qy_{\nu-1}}\ \cdots \int_{0}^{qy_3}\int_{0}^{qy_2} \frac{[\nu]_q!}{q^{\binom{\nu}{2}} t^\nu} d_qy_1\, d_qy_2 \cdots d_qy_{\nu-2}\,d_qy_{\nu-1}\,d_qy_\nu=1,
\end{align}
which confirms that Equation~\eqref{multjointorderunif} is a joint $q$-density function.
\end{proof}
 
\subsection{On a conditional joint \texorpdfstring{$q$}{q}-distribution of the waiting times of the Heine process and \texorpdfstring{$q$}{q}-order statistics }
 Let $T_k$ be the waiting time of the $k$th arrival in the Heine process $\{X(t),\,t>0 \}$ with parameters $\lambda$ and $q$. 
 Let us stop the process at $T_\nu$, for some integer $\nu\geq 1$.
 Now,  we study the joint $q$-density function of the waiting times $T_1,\ldots,T_{\nu}$.
  In the next theorem we prove that this conditional joint $q$-density function coincides with the joint $q$-density function of a $q$-ordered random sample of size $\nu$, 
 from the $q$-continuous uniform distribution in the set $[0,q^{i-1}t]$, $i=1,\ldots,\nu$.
\begin{theorem}
\label{basic1}
 Let $T_k$  be the waiting time of the $k$th arrival of the Heine process $\{X(t),\,t>0 \}$ with parameters $\lambda$ and $q$. Then the joint $q$-density function of the waiting times $T_1,\ldots,T_{\nu}$, 
in which the first $\nu$ events occur given that $X(t)=\nu$, $\,0<t_1<\cdots<t_\nu<t$ with $t_i \in (q^{\nu-i+1}t,q^{\nu-i}t ]$, $i=1,\ldots,\nu-1$, is given by
\begin{equation}
\label{orderunif}
f_q \left(t_1,\ldots,t_\nu|X(t)=\nu \right)=\frac{[\nu]_q!}{q^{\binom{\nu}{2}} t^\nu},
\end{equation}
that is the joint $q$-density function of a $q$-ordered random sample of size $\nu$, from the $q$-continuous uniform distribution in the set $[0,q^{i-1}t]$, $i=1,\ldots,\nu$.
 \end{theorem}
 \begin{proof}
 By using the expression~\eqref{multjointqdens} and the three basic assumptions of Definition~\ref{HeineProc},
 the conditional joint $q$-density function of the Heine process  satisfies the equation
 \begin{align}
&f_q \left(t_1,\ldots,t_\nu|X(t)=\nu \right)q^{\nu-1}(1-q)t q^{\nu-2}(1-q)t\cdots q(1-q)t\\
&= P\left(q^\nu t<T_1\leq q^{\nu-1} t,\ldots, q^2 t<T_{\nu}\leq qt \ \mid  \ X(t)=\nu \right) \\
&=P\left( X(q^\nu t)=0 \right)    \left(\prod_{i=1}^{\nu-1} P\left(X(q^{i}(1-q)t)=1\right)\right) \frac{P\left(X((1-q)t)=0\right)}{P\left( X(t)=\nu\right )}\\
&=e_q(-\lambda q^\nu t)\frac{\lambda q^{\nu-1}(1-q)t}{1+\lambda q^{\nu-1}(1-q)t}
\frac{\lambda q^{\nu-2}(1-q)t}{1+\lambda q^{\nu-2}(1-q)t}\cdots \frac{\lambda q(1-q)t}{1+\lambda q(1-q)t}\frac{1}{1+\lambda (1-q)t} \\
&\,\,\,\,\,\cdot \left(e_q(-\lambda t)\frac{ q^{\binom{\nu}{2}}(\lambda t)^\nu
}{[\nu]_q!}\right)^{-1}. 
\end{align}
So,
 \begin{equation}
 \label{orderunif2}
f_q \left(t_1,\ldots,t_\nu|X(t)=\nu \right)=\frac{[\nu]_q!}{q^{\binom{\nu}{2}} t^\nu}.
\end{equation}
Therefore, by Corollary~\ref{mutlquniform2}, this conditional joint $q$-density function 
coincides with $q$-ordered density from the claim of the theorem.
 \end{proof}
\bigskip

\section{Concluding remarks}   
In this  work we have introduced  $q$-order statistics, for $0<q<1$, arising from dependent and not identically $q$-continuous random variables, 
as $q$-analogues of the classical order statistics. 
We have  studied  their main properties concerning  the $q$-distribution functions and $q$-density functions of the relative $q$-ordered random variables. 
We have concentrated  on the $q$-ordered variables arising from dependent and not identically $q$-uniformly distributed random variables. 
The derived  $q$-distributions include $q$-power law, $q$-beta and $q$-Dirichlet distributions. 
The motivation for introducing $q$-order statistics was given by studying the properties of the waiting times of the Heine process.

 As further study we propose the introduction of $q$-order statistics arising from dependent and not identically discrete $q$-distributed random variables. 
Last but not least, in link with lattice paths combinatorics, we intend to  study the relations between $q$-order statistics  and $q$-random walks in ${\mathbb Z}^d$,
building on~\cite{VamKam}.
\bigskip

\section{Acknowledgment} The author would like to thank the anonymous referees and the editors for their helpful and insightful comments and suggestions.
\newpage

\bibliography{Vamvakari}
\bibliographystyle{SLC}
\end{document}